\documentclass{aic}

\aicAUTHORdetails{%
  title = {Clique Covers and Decompositions of Cliques of Graphs}, 
  author = {J\'ozsef Balogh, Jialin He, Robert A. Krueger, The Nguyen, and Michael C. Wigal},
  plaintextauthor = {József Balogh, Jialin He, Robert A. Krueger, The Nguyen, Michael C. Wigal},
  keywords = {graph covering, graph decomposition, Zykov symmetrization, nibble method, regularity lemma, LP relaxation},
}   

\aicEDITORdetails{%
   year={2026},
   number={9},
   received={20 February 2025},   
   published={2 October 2026},  
   doi={10.19086/aic/2026.9},      
}   

\usepackage{amsmath,amsthm,amssymb}
\usepackage{mathtools}
\usepackage{hyperref}
\usepackage{bm}
\usepackage{enumerate}
\usepackage{xcolor}
\usepackage{color}

\newtheorem{theorem}{Theorem}[section]
\newtheorem{lemma}[theorem]{Lemma}
\newtheorem{conjecture}[theorem]{Conjecture}

\newtheorem{proposition}[theorem]{Proposition}

\newcommand{\cliques}[1]{\K(#1)}
\newcommand{\tcliques}[2]{\K_{#1}(#2)}
\newcommand{\maximalcliques}[1]{\tcliques{\text{max}}{#1}}

\newcommand{\bc}{\mathbf{c}}

\newcommand{\cH}{\mathcal{H}}
\newcommand{\cX}{\mathcal{X}}

\newcommand{\R}{\mathbb{R}}

\newcommand{\tcoveropt}[1]{\mathcal{C}\mathcal{C}_t(#1)}
\newcommand{\tcoveroptfrac}[1]{\mathcal{C}\mathcal{C}^*_t(#1)}
\newcommand{\turangraph}[2]{T_{#1,#2}}
\newcommand{\coveropt}[2]{{\mathcal{C}\mathcal{C}_{#1}(#2)}}
\newcommand{\coveroptfrac}[2]{\mathcal{C}\mathcal{C}^*_{#1}(#2)}

\newcommand{\decompositionopt}[2]{\mathcal{C}\mathcal{D}_{#1}(#2)}
\newcommand{\decompositionoptfrac}[2]{\mathcal{C}\mathcal{D}^*_{#1}(#2)}

\newcommand{\floor}[1]{\left\lfloor#1\right\rfloor}
\newcommand{\ceil}[1]{\left\lceil#1\right\rceil}

\newcommand{\ep}{\varepsilon}
\newcommand{\K}{\mathcal{K}}
\newcommand{\norm}[1]{\left\lVert#1\right\rVert}

\newcommand{\pr}{\mathbb{P}}

\newcommand{\E}{\mathbb{E}}

\begin{document}

\begin{frontmatter}[classification=text]

\title{Clique Covers and Decompositions \\ of Cliques of Graphs} 

\author[jozsi]{J\'ozsef Balogh\thanks{Supported in part by NSF grants DMS-1764123 and RTG DMS-1937241, FRG DMS-2152488, the Arnold O.~Beckman Research Award (UIUC Campus Research Board RB 24012), and the Simons Fellowship.}}
\author[jialin]{Jialin He\thanks{Supported in part by Science and Technology Commission of Shanghai Municipality No. 22DZ2229014, Hong Kong RGC grant GRF 16308219, and Hong Kong RGC grant ECS 26304920.}}
\author[bob]{Robert A. Krueger\thanks{Supported by NSF Awards DGE 21-4675 and DMS-2402204.}}
\author[the]{The Nguyen}
\author[mike]{Michael C. Wigal\thanks{Supported in part by NSF RTG DMS-1937241 and an AMS-Simons Travel Grant.}}

\begin{abstract}
In 1966, Erd\H{o}s, Goodman, and P{\'o}sa showed that if $G$ is an $n$-vertex graph, then at most $\lfloor n^2/4 \rfloor$ cliques of $G$ are needed to cover the edges of $G$, and the bound is best possible as witnessed by the balanced complete bipartite graph. This was generalized independently by Gy{\H o}ri--Kostochka, Kahn, and Chung, who showed that every $n$-vertex graph admits an edge-decomposition into cliques of total `cost' at most $2 \lfloor n^2/4 \rfloor$, where an $i$-vertex clique has cost $i$. Erd\H{o}s suggested the following strengthening: every $n$-vertex graph admits an edge-decomposition into cliques of total cost at most $\lfloor n^2/4 \rfloor$, where now an $i$-vertex clique has cost $i-1$. We prove fractional relaxations and asymptotically optimal versions of both this conjecture and a conjecture of Dau, Milenkovic, and Puleo on covering the $t$-vertex cliques of a graph instead of the edges. Our proofs introduce a general framework for these problems using Zykov symmetrization, the Frankl--R\"odl nibble method, and the Szemer\'edi Regularity Lemma.
\end{abstract}
\end{frontmatter}


\section{Introduction}\label{sec:intro}

A \emph{$t$-clique}, denoted by $K_t$, is a complete graph on $t$ vertices. A \emph{$t$-clique cover (decomposition)} of a graph $G$ is a collection of cliques that `covers' (`partitions') all $t$-cliques of $G$, that is, every $t$-clique in $G$ is a subgraph of some (exactly one) clique from the clique cover (decomposition). A \emph{cost vector} $\mathbf{c}$ is a sequence of nonnegative real numbers $\mathbf{c} = (c_i)_{i=1}^\infty$. Let $\tcoveropt{G,\mathbf{c}}$ ($\decompositionopt{t}{G,\mathbf{c}}$) denote the lowest cost of a $t$-clique cover (decomposition) of $G$, where an $i$-clique has cost $c_i$, and the cost of a cover (decomposition) is the sum of the costs of the cliques in the cover (decomposition). When $\mathbf{c}$ is the constant $1$ cost vector, we write $\tcoveropt{G,\mathbf{c}} = \tcoveropt{G}$ ($\decompositionopt{t}{G,\mathbf{c}} = \decompositionopt{t}{G})$.

 It is easy to see that $\coveropt{1}{G}$ is the chromatic number of the complement of $G$. The graph parameter $\coveropt{2}{G}$ is well-studied~\cite{EGP,L,R,ST} and also known as the intersection number of $G$: $\coveropt{2}{G}$ is the minimum $k$ such that there exists $\phi: V(G) \to 2^{[k]}$ such that $uv \in E(G)$ if and only if $\phi(u) \cap \phi(v) \neq \emptyset$. Recently, Dau, Milenkovic, and Puleo~\cite{DMP} introduced $\coveropt{t}{G}$ for $t>2$, motivated by community detection problems.

We are interested in the following extremal question: what is the maximum value of $\coveropt{t}{G,\mathbf{c}}$ ($\decompositionopt{t}{G,\mathbf{c}}$) over all $n$-vertex graphs $G$? Our starting point is a result of Erd{\H o}s, Goodman, and P{\'o}sa~\cite{EGP}.
\begin{theorem}[Erd{\H o}s, Goodman, and P{\'o}sa~\cite{EGP}]\label{thm:EGP}
If $G$ is an $n$-vertex graph, then $\coveropt{2}{G} \leq \floor{n^2/4}$, and equality holds if and only if $G$ is a balanced complete bipartite graph. Furthermore, if $G$ is an $n$-vertex graph, then $\decompositionopt{2}{G,\mathbf{c}} \leq \floor{n^2/4}$, where $c_2 = c_3 = 1$ and $c_i = \infty$ for all $i \geq 4$.
\end{theorem}
In other words, it is the hardest to cover the edges of a graph with cliques when the graph is triangle-free and the number of edges is maximized with respect to this condition; furthermore, such a covering by cliques may be replaced by decomposition into edges and triangles. Bollob\'as~\cite{Bollobas}, proving a conjecture of Erd\H{o}s, extended Theorem~\ref{thm:EGP} and determined the maximum number of edges and $K_r$'s needed to decompose the edges of any $n$-vertex graph.

Perhaps the most natural non-constant cost vector is $c_i = i$ for all $i$. Katona and Tarj\'an conjectured the following generalization of Theorem~\ref{thm:EGP}.
\begin{theorem}[Gy\H{o}ri and Kostochka~\cite{GyoriKostochka}, Kahn~\cite{Kahn}, Chung~\cite{Chung}]\label{thm:GyKKC}
Let $c_i = i$ for all $i$. If $G$ is an $n$-vertex graph, then $\decompositionopt{2}{G, \mathbf{c}} \leq 2 \cdot \floor{n^2/4}$, and equality holds if and only if $G$ is a balanced complete bipartite graph.
\end{theorem}
Theorem~\ref{thm:GyKKC} generalizes Theorem~\ref{thm:EGP} in the following way. Observe that for $\mathbf{c}$ and $\mathbf{c'}$ with $c_i \leq c_i'$ for all $i \geq t$, we have that $\decompositionopt{t}{G,\mathbf{c}} \leq \decompositionopt{t}{G,\mathbf{c'}}$, $\coveropt{t}{G,\mathbf{c}} \leq \coveropt{t}{G,\mathbf{c'}}$, and $\coveropt{t}{G,\mathbf{c}} \leq \decompositionopt{t}{G,\mathbf{c}}$. Thus for the cost vector $c_i = i$ for all $i$ and the constant $1$ cost vector $\mathbf{1}$, we have $\coveropt{2}{G} 
 = \coveropt{2}{G,\mathbf{1}} = \frac{1}{2} \coveropt{2}{G,2\cdot \mathbf{1}} \leq \frac{1}{2} \decompositionopt{2}{G,\mathbf{c}}$, which is at most $\floor{n^2/4}$ by Theorem~\ref{thm:GyKKC}, implying the first part of Theorem~\ref{thm:EGP}. Later, McGuinness~\cite{McGuinness} showed that Theorem~\ref{thm:GyKKC} can be proved in a greedy manner. Recently, using the method of flag algebras and stability, in~\cite{BLPPPV,KLMP} the maximum value of $\decompositionopt{2}{G,\mathbf{c}}$ was determined for all $\mathbf{c}$ with $c_i = \infty$ for all $i \geq 4$, that is, when only edges and triangles are allowed in the decomposition. The case $c_2 = 2$ and $c_3 = 3$ is particularly challenging, as there are two different asymptotically extremal graphs: $K_{n/2,n/2}$ and $K_n$.

Erd\H{o}s suggested (see \cite[Problem 43]{Tuza} or \cite{Gyori}) the following strengthening of Theorem~\ref{thm:GyKKC}.
\begin{conjecture}[Erd\H{o}s]\label{conjerdos}
Let $c_i = i-1$ for all $i$. If $G$ is an $n$-vertex graph, then $\decompositionopt{2}{G,\mathbf{c}} \leq \floor{n^2/4}$.
\end{conjecture}
If true, Conjecture~\ref{conjerdos} is best possible, again as witnessed by the balanced complete bipartite graph. Conjecture~\ref{conjerdos} generalizes Theorem~\ref{thm:GyKKC} since 
\[ \decompositionopt{2}{G,(i)} = 2 \cdot \decompositionopt{2}{G,(i/2)} \leq 2 \cdot \decompositionopt{2}{G,(i-1)} ,\]
because $i/2 \leq i-1$ for all $i \geq 2$. While the cost vector $c_i = i-1$ may seem arbitrary, it has some connection to the well-studied spanning forest polytope --- see the discussion at the end of Section~\ref{sec:LP}. Conjecture~\ref{conjerdos} has been solved for $K_4$-free graphs by Gy\H{o}ri and Keszegh~\cite{GyoriKeszegh}, but otherwise remains open. We asymptotically solve Conjecture~\ref{conjerdos}.
\begin{theorem}\label{thm:erdos_decomp}
Let $c_i = i-1$ for all $i$. If $G$ is an $n$-vertex graph, then $\decompositionopt{2}{G,\mathbf{c}} \leq (1+o(1))n^2/4$, where the $o(1)$ goes to $0$ as $n$ goes to infinity.
That is, every $n$-vertex graph admits a partition of its edges into cliques $Q_1, Q_2, \dots$ such that $\sum_i (|V(Q_i)|-1) \leq (1+o(1)) n^2/4$.
\end{theorem}

We also study the extremal problem of maximizing $\coveropt{t}{G,\mathbf{c}}$ over all $n$-vertex graphs $G$ for $t > 2$. For $n \ge k$, denote by $\turangraph{n}{k}$ the \emph{Tur\'an graph}, the $n$-vertex complete $k$-partite graph with part sizes differing by at most one. Zykov~\cite{Z} generalized Tur\'an's Theorem~\cite{Turan} with a symmetrization argument, proving that the $K_k$-free $n$-vertex graph maximizing the number of $K_t$'s is $\turangraph{n}{k-1}$, provided $t < k$. Dau, Milenkovic, and Puleo~\cite{DMP} extended Theorem~\ref{thm:EGP} to $t=3$, proving that $\coveropt{3}{G} \leq \coveropt{3}{\turangraph{n}{3}}$, and (when $n \ge 3$) equality holds if and only if $G$ is isomorphic to $\turangraph{n}{3}$. They conjectured the following for larger $t$.
\begin{conjecture}[Dau, Milenkovic, and Puleo~\cite{DMP}]\label{conj:clique_cover}
For every $4\le t\le n$, for every $n$-vertex graph $G$, we have $\tcoveropt{G} \le \tcoveropt{\turangraph{n}{t}}$, with equality if and only if $G$ is isomorphic to $\turangraph{n}{t}$. 
\end{conjecture}

It is easy to calculate $\tcoveropt{\turangraph{n}{t}}$ explicitly, since the largest cliques in $\turangraph{n}{t}$ are $t$-cliques, and hence each must be taken in the cover. Dau, Milenkovic, and Puleo~\cite{DMP} proved Conjecture~\ref{conj:clique_cover} for $t=3$ by induction on $n$. 
They noted that for larger $t$, the number of cases in their proof would be too large to handle effectively. With a different approach, we prove an asymptotic version of Conjecture~\ref{conj:clique_cover}.

\begin{theorem}\label{thm:clique_cover_asym}
For every fixed $t \geq 2$, we have $\tcoveropt{G} \leq (1+o(1))\tcoveropt{\turangraph{n}{t}}$ for every $n$-vertex graph $G$.
\end{theorem}

We provide a general framework for proving theorems like Theorems~\ref{thm:erdos_decomp} and~\ref{thm:clique_cover_asym}, where the cost vectors could be chosen almost arbitrarily. We describe our approach in detail in the following subsection.

\subsection{Proof framework}

We first introduce some more notation. For a function $f : {\cal S} \to \mathbb{R}$, which we think of as an assignment of weights to the elements of $\cal S$, we let the \emph{total $\mathbf{c}$-cost} of $f$, denoted by $\norm{f}_\mathbf{c}$, be
\[ \norm{f}_\mathbf{c} = \sum_{S \in {\cal S}} c_{|S|}f(S) .\]
The \emph{support} of $f$ is the set of $S \in {\cal S}$ with $f(S) \neq 0$.

For a graph $G$, let $\cliques{G}$ denote the set of  cliques of $G$,  $\tcliques{\le t}{G}$  the set of  cliques of $G$ of size at most $t$, and  $\tcliques{t}{G}$  the set of  cliques of $G$ of size $t$. We may view a $t$-clique cover of $G$ as an assignment of weights $f : \cliques{G} \to \{0,1\}$ such that for every $T \in \tcliques{t}{G}$,
\begin{equation}\label{eq:covercond}
\sum_{\substack{T \subseteq K \in \cliques{G} }} f(K) \ge 1.
\end{equation}
In this formulation, for a cost vector $\mathbf{c}$, $\tcoveropt{G,\mathbf{c}}$ is the minimum total $\mathbf{c}$-cost of a $t$-clique cover of $G$. A \emph{fractional $t$-clique cover} is an assignment of \emph{nonnegative real} weights $f : \cliques{G} \to \mathbb{R}_{\ge 0}$ satisfying~\eqref{eq:covercond}. We let $\tcoveroptfrac{G,\mathbf{c}}$ denote the minimum total $\mathbf{c}$-cost of a fractional $t$-clique cover of $G$. By definition, $\tcoveroptfrac{G,\mathbf{c}} \le \tcoveropt{G,\mathbf{c}}$ for all $G$ and $\mathbf{c}$.

A \emph{(fractional) $t$-clique decomposition} is an assignment of weights $f: \cliques{G} \to \{0,1\}$ ($f : \cliques{G} \to \mathbb{R}_{\geq 0}$) such that equality holds in~\eqref{eq:covercond} for every $T \in \tcliques{t}{G}$. We let $\decompositionopt{t}{G,\mathbf{c}}$ be the minimum total $\mathbf{c}$-cost of a $t$-clique decomposition of $G$, and similarly, we let $\decompositionoptfrac{t}{G,\mathbf{c}}$ be the minimum total $\mathbf{c}$-cost of a fractional $t$-clique decomposition of $G$.

Our first general main theorem shows that the maximizer, at least for the fractional version, is always a complete multipartite graph.

\begin{theorem}\label{thm:frac}
Let $G$ be an $n$-vertex graph,  $t \ge 2$, and $\mathbf{c} = (c_i)_{i = 1}^{\infty}$ be a  sequence of nonnegative real numbers. Then there exist $n$-vertex complete multipartite graphs $H, H'$ such that
\[ \coveroptfrac{t}{G,\mathbf{c}} \le \coveroptfrac{t}{H,\mathbf{c}} \]
and
\[ \decompositionoptfrac{t}{G,\mathbf{c}} \le \decompositionoptfrac{t}{H',\mathbf{c}} .\]
\end{theorem}

The proof of Theorem~\ref{thm:frac} relies on the following symmetrization operation. Given a graph $G$ and nonadjacent vertices $u,v \in V(G)$, we \emph{symmetrize} $u$ to $v$ by changing $N(u)$ to $N(v)$. A \emph{clone} of a vertex $u$ in $G$ is a vertex $v \in V(G)$ such that $u \neq v$, $uv \not \in E(G)$, and $N(u) = N(v)$. Note that after \emph{symmetrizing} $u$ to $v$, $u$ and $v$ are clones of each other. Let $U = \{u_1, \ldots, u_i\}$ be a set of pairwise clones of $G$ and $V = \{v_1,\ldots, v_j\}$ also be a set of pairwise clones of $G$ such that $u_1v_1 \not \in E(G)$. We say we \emph{symmetrize} $U$ to $V$ by deleting every edge incident with the vertices of $U$, and adding edges to $G$ such that $N(u_k) = N(v_1)$ for all $1 \le k \le i$. 

Zykov's proof~\cite{Z} of a generalization of Tur\'an's Theorem~\cite{Turan} shows that from every $K_{k+1}$-free graph $G$, one can obtain a $k$-partite graph on the same vertex set by a sequence of symmetrization operations such that during the entire process the number of $K_t$'s never decreases. This shows that $G$ has at most as many $K_t$'s as some complete $k$-partite graph, and among those graphs $\turangraph{n}{k}$ maximizes the number of $K_t$'s. We take a similar approach, reducing the fractional problem to a weighted version on cliques using Zykov symmetrization to prove Theorem~\ref{thm:frac}.

Using the symmetry of complete multipartite graphs, Theorem~\ref{thm:frac} reduces the number of constraints a fractional $t$-clique cover/decomposition must satisfy from the number of $t$-cliques of the graph to the number of $t$-sets of parts in the vertex partition. Unfortunately, no effective bound can be given on the number of parts. However, the problem of maximizing $\coveroptfrac{t}{H,\mathbf{c}}$ or $\decompositionoptfrac{t}{H,\mathbf{c}}$ over all $n$-vertex complete multipartite graphs $H$ has a particularly nice description as a set of linear programs (see Section~\ref{sec:LP}), which we  solve in \emph{ad hoc} ways. We do this for the cost vectors which are the subjects of Theorems~\ref{thm:erdos_decomp} and~\ref{thm:clique_cover_asym}.

\begin{theorem}\label{thm:solve_LP}
Let $H$ be an $n$-vertex complete multipartite graph, let $t \geq 2$, and let $\mathbf{c} = (c_i)_{i = 1}^{\infty}$ be a sequence of nonnegative real numbers.
\begin{enumerate}[i.]
    \item If $\mathbf{c}$ is such that $c_{i+1} - c_i \leq c_t$ for every $i \geq t$, then
\[ \coveroptfrac{t}{H,\mathbf{c}} \leq \coveroptfrac{t}{\turangraph{n}{t},\mathbf{c}} = \coveropt{t}{T_{n,t},\mathbf{c}} .\]
    \item If $\mathbf{c}$ is such that $2c_i \geq c_{i-1} + c_{i+1}$ for every $i \geq 2$, then
\[ \decompositionoptfrac{2}{H,\mathbf{c}} \leq \decompositionoptfrac{2}{\turangraph{n}{2},\mathbf{c}} = \decompositionopt{2}{\turangraph{n}{2},\mathbf{c}}. \]
\end{enumerate}
\end{theorem}

The conditions $c_{i+1} - c_i \le c_t$ and $2c_i \geq c_{i-1} + c_{i+1}$ are an artifact of the proof of Theorem~\ref{thm:solve_LP}. In particular, both conditions are satisfied when $c_i$ is constant, when $c_i = i$, and when $c_i = i-1$. Both conditions say that $\mathbf{c}$ cannot grow too quickly, which is a natural requirement: if a particular clique size is too expensive, then it would not be used. The first part of Theorem~\ref{thm:solve_LP} is proven with a fairly straightforward greedy approach, but the second part requires a fairly delicate induction on the number of parts --- see Lemma~\ref{lem:decompLP}. It is perhaps surprising that even the fractional version of Conjecture~\ref{conjerdos} appears difficult even for complete multipartite graphs.

Our second general main result relates the fractional cover problems to the integer ones.
\begin{theorem}\label{thm:asymptotic_bound}
Let $G$ be an $n$-vertex graph, $t \ge 2$ be an integer, and let $\mathbf{c} = (c_i)_{i = 1}^{\infty}$ be a sequence of nonnegative real numbers such that $c_i \ge 1$ for all $i \ge t$. Then for every $\rho > 0$ there exists an $n_0(\rho,c_t)$, such that for every $n > n_0(\rho,c_t)$ there exists an $n$-vertex graph $\tilde{G}$ such that
\[ \coveropt{t}{G,\mathbf{c}} \le \coveroptfrac{t}{\tilde{G},\mathbf{c}} + \rho n^t \]
and
\[ \decompositionopt{t}{G,\mathbf{c}} \le \decompositionoptfrac{t}{\tilde{G},\mathbf{c}} + \rho n^t. \]
\end{theorem}
We prove Theorem~\ref{thm:asymptotic_bound} using  Szemer\'edi's Regularity Lemma~\cite{Sz} and Pippenger's formulation~\cite{AS,F} of the Frankl--R\"odl nibble~\cite{FR}. The approach is similar to what Haxell and R\"odl~\cite{HR} and Yuster~\cite{Y} did in the context of \emph{packing} edge-disjoint copies of a graph into another graph. For clique covers, additional difficulties exist --- see the remark at the end of Section~\ref{sec:asymptotic:frac_to_int}.

Theorems~\ref{thm:erdos_decomp} and~\ref{thm:clique_cover_asym} can now be read as corollaries of Theorems~\ref{thm:frac},~\ref{thm:solve_LP}, and~\ref{thm:asymptotic_bound}. 

\noindent 
\begin{proof}[Proof of Theorem~\ref{thm:erdos_decomp}.]
Using Theorems~\ref{thm:asymptotic_bound},~\ref{thm:frac} and~\ref{thm:solve_LP} (ii), for sufficiently large $n$, for every $n$-vertex graph there exists an $n$-vertex graph $\tilde{G}$ and an $n$-vertex complete multipartite graph $H$ such that
\[
\decompositionopt{2}{G,\mathbf{c}} \le \decompositionoptfrac{2}{\tilde{G},\mathbf{c}} + \rho n^2\le 
\decompositionoptfrac{2}{H,\mathbf{c}} + \rho n^2 \le  \decompositionopt{2}{T_{n,2},\mathbf{c}} + \rho n^2 
,\]
where $\rho>0$ is an arbitrarily small constant and $c_i=i-1$ for all $i$.
\end{proof}

\begin{proof}[Proof of Theorem~\ref
{thm:clique_cover_asym}.]
Let $\mathbf{1}$ denote the constant $1$ vector. Using Theorems~\ref{thm:asymptotic_bound},~\ref{thm:frac},
and~\ref{thm:solve_LP} (i), for sufficiently large $n$, for every $n$-vertex graph $G$ there exists an $n$-vertex graph $\tilde{G}$ and an $n$-vertex complete multipartite graph $H$ such that
\[ \coveropt{t}{G,\mathbf{1}} \le
\coveroptfrac{t}{\tilde{G},\mathbf{1}} + \rho n^t \le 
\coveroptfrac{t}{H,\mathbf{1}}  + \rho n^t  \le \coveropt{t}{T_{n,t},\mathbf{1}} + \rho n^t,\]
where $\rho>0$ is an arbitrarily small constant.
\end{proof}

Our third main result concerns the (additive) integrality gap for clique decompositions. In combinatorial optimization, a natural question is how different an integer program is from its linear relaxation. For decompositions, we show that $\decompositionopt{t}{G,\mathbf{c}}$ is not (additively) far from $\decompositionoptfrac{t}{G,\mathbf{c}}$.
\begin{theorem}\label{thm:asymptotic_bound_decomposition}
Let $G$ be an $n$-vertex graph, $t \ge 2$ be an integer, and let $\mathbf{c} = (c_i)_{i = 1}^{\infty}$ be a sequence of nonnegative real numbers such that $c_i \ge 1$ and $c_{i+1} - c_i \le c_t$ for all $i \ge t$. Then for every $\rho > 0$ there exists an $n_0(\rho,c_t)$, such that for every $n > n_0(\rho,c_t)$
\[ \decompositionopt{t}{G,\mathbf{c}} \le \decompositionoptfrac{t}{G,\mathbf{c}} + \rho n^t .\]
\end{theorem}
We conjecture (see Conjecture~\ref{conj:gap}) that Theorem~\ref{thm:asymptotic_bound_decomposition} holds for covers in place of decompositions. We discuss these integrality gaps in Section~\ref{sec:gap}, where, among other things, we show that the $o(n^t)$ error in Theorems~\ref{thm:asymptotic_bound} and~\ref{thm:asymptotic_bound_decomposition} cannot be improved beyond $n^{2-o(1)}$. Our construction applied to the ``packing'' problem answers a question of Yuster~\cite{Y}.

 The rest of the paper is organized as follows. Theorem~\ref{thm:frac} is proven in Section~\ref{sec:frac}, Theorem~\ref{thm:solve_LP} is proven in Section~\ref{sec:LP}, and Theorem~\ref{thm:asymptotic_bound} is proven in Section~\ref{sec:asymptotic}. 
 In Section~\ref{sec:gap} we discuss the integrality gap, proving Theorem~\ref{thm:asymptotic_bound_decomposition}.

\section{Symmetrization: proof of Theorem~\ref{thm:frac}}\label{sec:frac}

Recall that $U \subseteq V(G)$ is a set of pairwise clones in $G$ if $U$ is nonempty and $N(u)$ is the same for all $u \in U$. For sets of pairwise clones $V_0$ and $V_1$ with no edges between $V_0$ and $V_1$, we let $G(V_0 \to V_1)$ denote the graph obtained from $G$ by \emph{symmetrizing} $V_0$ to $V_1$, where we delete all the edges incident to $V_0$ and make all the neighbors of $V_1$ to be the neighborhoods of the vertices in $V_0$. More precisely, $V(G(V_0 \to V_1)) = V(G)$ and
\[ E(G(V_0 \to V_1)) = \{ xy \in E(G) : x,y \not\in V_0 \cup V_1 \} \cup \{ xy : x \in V_0 \cup V_1, y \not\in V_0 \cup V_1, zy \in E(G) \}, \]
where $z \in V_1$ is arbitrary (since $V_1$ is a set of clones, the choice of $z \in V_1$ does not matter for this definition).

The following lemma shows that at least one of the symmetrizations, either $V_0$ to $V_1$ or $V_1$ to $V_0$, does not decrease the  fractional $t$-clique cover/decomposition number.

\begin{lemma}\label{lem:frac_sym}
Let $\mathbf{c}$ be a nonnegative real sequence. Let $V_0$ and $V_1$ be sets of pairwise clones in a graph $G$. For every integer $t \geq 1$, we have
\[ \tcoveroptfrac{G,\mathbf{c}} \leq \max\{\tcoveroptfrac{G(V_0\to V_1),\mathbf{c}}, \ \tcoveroptfrac{G(V_1\to V_0),\mathbf{c}} \} \]
and
\[ \decompositionoptfrac{t}{G,\mathbf{c}} \leq \max\{\decompositionoptfrac{t}{G(V_0\to V_1),\mathbf{c}}, \ \decompositionoptfrac{t}{G(V_1\to V_0),\mathbf{c}} \} .\]
\end{lemma}

\begin{proof}
We prove the lemma for the cover number and note the small changes needed to prove the lemma for the decomposition number.

Observe that $V_0 \cup V_1$ is an independent set in $G$, $G(V_0 \to V_1)$, and $G(V_1 \to V_0)$. Also observe that, since $V_0 \cup V_1$ is a set of pairwise clones in $G(V_0 \to V_1)$ and $G(V_1 \to V_0)$, any permutation of $V(G)$ fixing $V(G) \setminus (V_0 \cup V_1)$ is an automorphism of $G(V_0 \to V_1)$ and $G(V_1 \to V_0)$. Let $\mathcal{S}$ be the set of all these permutations.

For $i \in \{0,1\}$, let $f_i$ be an optimal $\mathbf{c}$-cost fractional $t$-clique cover  of $G(V_{1-i} \to V_i)$. Let $\bar{f}_i$ be the average of $f_i \circ \sigma$ over all permutations $\sigma \in \mathcal{S}$. Since each $f_i \circ \sigma$ is an optimal $\mathbf{c}$-cost fractional $t$-clique cover of $G(V_{1-i} \to V_i)$, it is easy to check that $\bar{f}_i$ is one as well. Because of this averaging, we have that $\bar{f}_i(K \cup \{v\})$ does not depend on $v \in V_0 \cup V_1$ for cliques $K$ in the neighborhood of $V_i$. 

Define $f: \cliques{G} \to \mathbb{R}_{\ge 0}$ as follows:
\begin{align*} f(K) = \begin{cases} \frac{|V_0|}{|V_0|+|V_1|} \bar{f}_0(K) + \frac{|V_1|}{|V_0|+|V_1|} \bar{f}_1(K), &\quad |K \cap (V_0 \cup V_1)| = 0, \\ \bar{f}_0(K),  &\quad |K \cap V_0| = 1,  \\ 
\bar{f}_1(K), &\quad |K \cap V_1| = 1, \end{cases} \end{align*}
which is well-defined for cliques $K \in \cliques{G}$ since $V_0 \cup V_1$ is an independent set, and hence $|K \cap (V_0 \cup V_1)| \leq 1$. We claim that $f$ is a fractional $t$-clique cover of $G$. Indeed, for $T \in \K_t(G)$ with $|T \cap V_i| = 1$ for some $i \in \{0,1\}$, we have 
\[ \sum_{\substack{K \in \K(G) \\ T \subseteq K}} f(K) = \sum_{\substack{K \in \K(G) \\ T \subseteq K}} \bar{f}_i(K) = \sum_{\substack{K \in \K(G(V_{1-i}\to V_i)) \\  T \subseteq K}} \bar{f}_i(K) \geq 1,\]
since the set of cliques $K$ in the sum is the same in $\K(G)$ and $\K(G(V_{1-i}\to V_i))$. (For decompositions instead of covers, the final inequality above would be equality.) For $T \in \K_t(G)$ with $|T \cap (V_0 \cup V_1)| = 0$, we have
\begin{equation}\label{eq:symmetrization}
\sum_{\substack{K \in \K(G) \\ T \subseteq K}} f(K) = \sum_{i \in \{0,1\}} \Big( \sum_{\substack{T\subseteq K \in \K(G) \\ |K \cap (V_0 \cup V_1)| = 0 }} \frac{|V_i|}{|V_0|+|V_1|} \bar{f}_i(K) + \sum_{\substack{T \subseteq K \in \K(G) \\ |K \cap V_i| = 1 }} \bar{f}_i(K) \Big). 
\end{equation}
Let $K$ be a clique in the neighborhood of $V_i$. Recall, for $v \in V_0 \cup V_1$, as $\bar{f}_i(K \cup \{v\})$ does not depend on $v$, for $i \in \{0,1\}$ we have 
\begin{align*}
    \sum_{\substack{K \in \K(G) \\ |K \cap V_i| = 1 \\ T \subseteq K}} \bar{f}_i(K) =  \frac{|V_i|}{|V_0| + |V_1|}  \sum_{\substack{K \in \K(G) \\ |K \cap V_i| = 1 \\ T \subseteq K}} \left(1 + \frac{|V_{1-i}|}{|V_i|}\right)\bar{f}_i(K) = \frac{|V_i|}{|V_0| + |V_1|}  \sum_{\substack{K \in \K(G(V_{1-i}\to V_i)) \\ |K \cap (V_0 \cup V_1)|=1 \\ T \subseteq K}} \bar{f}_i(K).
\end{align*}
Continuing \eqref{eq:symmetrization}, it follows,
\begin{align*}
\sum_{\substack{K \in \K(G) \\ T \subseteq K}} f(K) &= \sum_{i \in \{0,1\}}\frac{|V_i|}{|V_0|+|V_1|}  \Big( \sum_{\substack{K \in \K(G(V_{1-i} \to V_i)) \\ |K \cap (V_0 \cup V_1)| = 0 \\ T \subseteq K}} \bar{f}_i(K) + \sum_{\substack{K \in \K(G(V_{1-i}\to V_i)) \\ |K \cap (V_0 \cup V_1)|=1 \\ T \subseteq K}} \bar{f}_i(K) \Big) \\
&= \sum_{i \in \{0,1\}} \Big( \frac{|V_i|}{|V_0|+|V_1|} \sum_{\substack{K \in \K(G(V_{1-i}\to V_i)) \\ T \subseteq K}} \bar{f}_i(K) \Big) 
\geq \sum_{i \in \{0,1\}} \frac{|V_i|}{|V_0|+|V_1|} = 1 .
\end{align*}
(Again, for decompositions instead of covers, the inequality above would be equality.)
A similar calculation bounds the $\mathbf{c}$-cost of $f$:  
\begin{align*}
    \norm{f}_{\mathbf{c}} &=  \sum_{K \in \K(G)} c_{|K|} f(K) = \sum_{\substack{K \in \cliques{G}\\ |K \cap (V_0 \cup V_1)| = 0}} c_{|K|}f(K) + \sum_{\substack{K \in \cliques{G}\\ |K \cap (V_0 \cup V_1)| = 1}} c_{|K|}f(K)\\
    &= \sum_{i \in \{0,1\}}  \Big( \sum_{\substack{K \in \cliques{G}\\ |K \cap (V_0 \cup V_1)| = 0}} c_{|K|} \frac{|V_i|}{|V_0|+|V_1|} \bar{f}_i(K) + \sum_{\substack{K \in \cliques{G}\\ |K \cap V_i| = 1}}  c_{|K|}\bar{f}_i(K) \Big)\\
    &= \sum_{i \in \{0,1\}} \Big( \frac{|V_i|}{|V_0|+|V_1|} \sum_{K \in \K(G(V_{1-i}\to V_i))}  c_{|K|}\bar{f}_i(K)\Big) 
    = \frac{|V_0|}{|V_0|+|V_1|} \norm{\bar{f}_0}_{\mathbf{c}} + \frac{|V_1|}{|V_0|+|V_1|} \norm{\bar{f}_1}_{\mathbf{c}}.
\end{align*}
Thus
\begin{align*}
 \tcoveroptfrac{G,\bc} \ \leq  & \ \frac{|V_0|}{|V_0|+|V_1|} \tcoveroptfrac{G(V_1\to V_0),\mathbf{c}} + \frac{|V_1|}{|V_0|+|V_1|} \tcoveroptfrac{G(V_0\to V_1),\mathbf{c}} \\
 \leq & \ \max\{\tcoveroptfrac{G(V_0\to V_1),\mathbf{c}}, \tcoveroptfrac{G(V_1\to V_0),\mathbf{c}} \} .\qedhere
 \end{align*}
\end{proof}

By repeatedly symmetrizing, we eventually obtain a complete multipartite graph, yielding Theorem~\ref{thm:frac}.
    
\begin{proof}[Proof of Theorem~\ref{thm:frac}]
Observe that the maximal sets of pairwise clones are the equivalence classes of $\sim$ on $V(G)$, where $u \sim v$ if and only if $N(u) = N(v)$. If there exist disjoint maximal sets of pairwise clones $V_0$ and $V_1$ in $G$ with no edges between $V_0$ and $V_1$, then by Lemma~\ref{lem:frac_sym}, at least one of  the symmetrizations,  $V_0$ to $V_1$ or $V_1$ to $V_0$ does not decrease the fractional $t$-clique cover/decomposition number of the graph. At the same time, this symmetrization decreases the number of maximal sets of pairwise clones, since vertices which were previously clones remain clones. Repeat this symmetrization process until $V(G)$ is partitioned into maximal sets $V_1, \dots, V_k$ of pairwise clones for some $k$, where there is an edge between every pair of $V_i$ and $V_j$ with $i \neq j$. Since the $V_i$'s are sets of pairwise clones, this final graph is a complete $k$-partite graph.
\end{proof}

\section{Solving linear programs: proof of Theorem~\ref{thm:solve_LP}}\label{sec:LP}

In this section, we aim to maximize $\coveroptfrac{t}{H,\mathbf{c}}$ and $\decompositionoptfrac{t}{H,\mathbf{c}}$ over all $n$-vertex complete multipartite graphs $H$.
We first show that, through appropriate normalization, this problem is equivalent to the following collections of linear programs (LPs).

Let $k \ge 2$ and $x_1, \dots, x_k$ be nonnegative reals such that $\sum_i x_i = 1$. These are the parameters of our LPs; here, $k$ is the number of parts of the complete multipartite graph $H$, and $x_i$ is the fraction of vertices in part $i$, so naturally $\sum_i x_i = 1$. Since every clique of $H$ meets each part in at most one vertex, the type of a clique can be identified with the set $I\subseteq[k]$ of parts that it meets. The variables of our LPs are $f(I)$, representing the total weight assigned to cliques of type $I$, for $f: 2^{[k]} \to \mathbb{R}_{\geq 0}$ and $I \subseteq [k]$.

For the cover problem, the LP in question, parameterized by $x_1, \dots, x_k$, is
\begin{equation}\label{LP:cover}\tag{CCLP}
\begin{array}{rll}
\text{minimize} & \norm{f}_\mathbf{c} = \displaystyle\sum_{I \subseteq [k]} c_{|I|} f(I) & \\
\text{subject to} & \displaystyle\sum_{I \supseteq T} f(I) \geq \prod_{i \in T} x_i, &\forall T \in \binom{[k]}{t}, \\
& f(I) \geq 0, &\forall I \subseteq [k].
\end{array}
\end{equation}
For the decomposition problem, the LP in question is
\begin{equation}\label{LP:decomp}\tag{CDLP}
\begin{array}{rll}
\text{minimize} & \norm{f}_\mathbf{c} = \displaystyle\sum_{I \subseteq [k]} c_{|I|} f(I) & \\
\text{subject to} & \displaystyle\sum_{I \supseteq T} f(I) = \prod_{i \in T} x_i, &\forall T \in \binom{[k]}{t}, \\
& f(I) \geq 0, &\forall I \subseteq [k].
\end{array}
\end{equation}
These LPs are the cover/decomposition problems for weighted complete graphs, where the weights on the $t$-cliques are derived from weights on the vertices. Note that these LPs are feasible by taking $f(T) = \prod_{i \in T} x_i$ for all $T \in \binom{[k]}{t}$ and all other $f(I) = 0$, and they are also bounded.
We now show the equivalence of these LPs to maximizing $\coveroptfrac{t}{H,\mathbf{c}}$ and $\decompositionoptfrac{t}{H,\mathbf{c}}$ over all $n$-vertex complete multipartite graphs $H$.

\begin{lemma}\label{lem:LP_reduction}
Let $\mathbf{c}$ be a nonnegative sequence, and let $t \geq 2$. Let $H$ be an $n$-vertex complete multipartite graph with parts $V_1, \dots, V_k$, and let $x_i = |V_i|/n$ for $i \in [k]$. If $f$ is an optimal solution to~\eqref{LP:cover} with parameters $x_1, \dots, x_k$, then $\coveroptfrac{t}{H,\mathbf{c}} = n^t \norm{f}_\mathbf{c}$. Similarly, if $f$ is an optimal solution to~\eqref{LP:decomp} with parameters $x_1, \dots, x_k$, then $\decompositionoptfrac{t}{H,\mathbf{c}} = n^t \norm{f}_\mathbf{c}$.
\end{lemma}

\begin{proof}
We prove the lemma for the cover number and note the small changes needed to prove the lemma for the decomposition number.

Let $f$ be a solution to~\eqref{LP:cover} with parameters $x_1, \dots, x_k$. We create a fractional $t$-clique cover $g$ of $H$ as follows: for each $K \in \cliques{H}$, let $g(K) = n^t f(J) / \prod_{j \in J} |V_j|$, where $J = J(K) = \{j \in [k] : K \cap V_j \neq \emptyset\}$. Let $\tilde{T} \in \tcliques{t}{H}$ and $T = T(\tilde{T}) = \{i \in [k] : \tilde{T} \cap V_i \neq \emptyset\}$. We define an empty product to be  $1$, as usual. The total weight of the cliques of $H$ containing $\tilde{T}$ is
\[ \sum_{I \supseteq T} n^t \frac{f(I)}{\prod_{i \in I} |V_i|} \prod_{\ell \in I \setminus T} |V_\ell| = n^t \sum_{I \supseteq T} \frac{f(I)}{\prod_{i \in T} |V_i|} \geq n^t \prod_{i \in T} \frac{x_i}{|V_i|} = 1 ,\]
where the inequality comes from~\eqref{LP:cover}, and would be replaced by equality of~\eqref{LP:decomp} for the decomposition number. Thus $g$ is a fractional $t$-clique cover of $H$. The total $\mathbf{c}$-cost of $g$ is
\[ \norm{g}_\mathbf{c} = \sum_{I \subseteq [k]} c_{|I|} n^t \frac{f(I)}{\prod_{i \in I} |V_i|} \prod_{i \in I} |V_i| = n^t \norm{f}_{\mathbf{c}} ,\]
and hence $\coveroptfrac{t}{H,\mathbf{c}} \leq n^t \norm{f}_\mathbf{c}$.

For the other direction, let $g$ be a fractional $t$-clique cover of $H$. We create a feasible solution to~\eqref{LP:cover} as follows: for each $I \subseteq [k]$, let $H(I) = \{ K \in \cliques{H} : K \cap V_i \neq \emptyset \text{ iff } i \in I\}$, and set
\[ f(I) = \frac{1}{n^t} \sum_{K \in H(I)} g(K) .\]
Then for every $T \in \binom{[k]}{t}$, we have
\[ \sum_{I \supseteq T} f(I) = \sum_{I \supseteq T} \frac{1}{n^t} \sum_{K \in H(I)} g(K) = \sum_{\tilde{T} \in H(T)} \frac{1}{n^t} \sum_{I \supseteq T} \sum_{\substack{K \in H(I) \\ \tilde{T} \subseteq K}} g(K) \geq \sum_{\tilde{T} \in H(T)} \frac{1}{n^t} = \prod_{i \in T} x_i ,\]
where the inequality comes from the definition~\eqref{eq:covercond} of a $t$-clique cover and would be replaced by equality for decompositions. The total $\mathbf{c}$-cost of $g$ is
\[ \norm{f}_\mathbf{c} = \sum_{I \subseteq [k]} c_{|I|} f(I) = \frac{1}{n^t} \sum_{K \in \cliques{H}} c_{|K|} g(K) = \frac{1}{n^t} \norm{g}_\mathbf{c} ,\]
and hence the optimal value of~\eqref{LP:cover} is at most $\frac{1}{n^t} \coveroptfrac{t}{H,\mathbf{c}}$.
The decomposition case is identical, with each covering inequality replaced by equality. Hence the same argument gives the CDLP statement.
\end{proof}

Lemma~\ref{lem:LP_reduction} tells us that in order to prove the bounds $\coveroptfrac{t}{H} \leq \eta n^t$ or $\decompositionoptfrac{t}{H} \leq \eta n^t$ for all complete multipartite $H$, it suffices to find feasible solutions $f$ to~\eqref{LP:cover} or~\eqref{LP:decomp} for all parameters $x_1, \dots, x_k$ such that $\norm{f}_\mathbf{c} \leq \eta$. We do this first for the cover number, corresponding to the first part of Theorem~\ref{thm:solve_LP}.

\begin{lemma}\label{lem:coverLP}
Let $k\geq t\geq 2$ be integers, and let
$x_1,\dots,x_k$ be nonnegative real numbers satisfying $x_1\geq x_2\geq\cdots\geq x_k$
and $\sum_{i=1}^k x_i=1.$
Let $\mathbf{c}$ be a sequence of nonnegative reals such that $c_{i+1} - c_i \leq c_t$ for all $i \geq t$. Then there exists a feasible solution $f$ to~\eqref{LP:cover} such that 
\[ \norm{f}_\mathbf{c} \leq c_t \left( \prod_{i=1}^{t-1} x_i \right) \left( 1 - \sum_{i=1}^{t-1} x_i \right) .\]
\end{lemma}

\begin{proof}
Define $f: 2^{[k]} \to \mathbb{R}_{\geq 0}$ by
\[ f(\{1,2,\dots,j\}) = \left( x_j - x_{j+1} \right) \prod_{i=1}^{t-1} x_i \]
for $t \leq j \leq k$, where we let $x_{k+1} = 0$, and $f(I) = 0$ for all other $I$. We first show that $f$ is a feasible solution to~\eqref{LP:cover}. For $T \in \binom{[k]}{t}$ with $m = \max(T)$, we have
\[ \sum_{I \supseteq T} f(I) = \sum_{j=m}^k f(\{1,2,\dots,j\}) = \sum_{j=m}^k \left( x_j - x_{j+1} \right) \prod_{i=1}^{t-1} x_i = x_m \prod_{i=1}^{t-1} x_i \geq \prod_{i \in T} x_i ,\]
where the last inequality follows from $x_1 \geq x_2 \geq \cdots \geq x_k$, which also shows that $f(I) \geq 0$ for all $I$. Thus $f$ is feasible. Using $c_j - c_{j-1} \leq c_t$ for all $j > t$ and $\sum_i x_i = 1$, we have
\begin{align*}
    \norm{f}_\mathbf{c} &= \sum_{j=t}^k c_j \left( x_j - x_{j+1} \right) \prod_{i=1}^{t-1} x_i = \left( \prod_{i=1}^{t-1} x_i \right) \left( c_t x_t + \sum_{j=t+1}^k x_j (c_j - c_{j-1}) \right)\\
    &\leq \left(\prod_{i=1}^{t-1} x_i \right)\sum_{j=t}^k c_t x_j = c_t\left( \prod_{i=1}^{t-1} x_i \right) \left( 1 - \sum_{i=1}^{t-1} x_i \right).\qedhere
\end{align*}
\end{proof}

It is more difficult to find solutions to~\eqref{LP:decomp} than~\eqref{LP:cover} because of the strict equality conditions. With more work, we are able to give a good enough bound on~\eqref{LP:decomp} to imply Theorem~\ref{thm:solve_LP} (ii).

\begin{lemma}\label{lem:decompLP}
Let $\mathbf{c}$ be a sequence of nonnegative reals such that $2c_i \geq c_{i-1} + c_{i+1}$ for all $i \geq 3$. Then when $t=2$, there exists a feasible solution to~\eqref{LP:decomp} such that
\[ \norm{f}_\mathbf{c} \leq c_2 \cdot \max_{J \subseteq [k]} \left( \sum_{j \in J} x_j \right) \cdot \left( \sum_{j \not\in J} x_j \right) .\]
\end{lemma}

\begin{proof}

We proceed by induction on $k$. For $k=2$, the optimal solution $f$ to~\eqref{LP:decomp} is trivially $f(\{1,2\}) = x_1 x_2$ and $0$ elsewhere. Since $x_1 + x_2 = 1$, we have $\norm{f}_\mathbf{c} = c_2 x_1(1-x_1)$.

Suppose $k \geq 3$, and let $x_1 \geq \ldots \geq x_k$ be given. We begin with a $g: 2^{[k-1]} \to \mathbb{R}_{\geq 0}$ such that for every $i,j \in [k-1]$ with $i<j$, we have
\[ \sum_{I \ni i,j} g(I) = \begin{cases} x_i x_j & \text{ if } j \neq k-1, \\ x_i (x_{k-1} + x_k) & \text{ if } j = k-1. \end{cases} \]
That is, we have a feasible solution to~\eqref{LP:decomp} when $x_{k-1}$ and $x_k$ are merged; in the formulation of complete multipartite graphs, $g$ corresponds to a fractional clique-decomposition of the graph where we delete all the edges between the smallest two parts. Using $g$, we write down a feasible solution to~\eqref{LP:decomp} with $\mathbf{c}$-cost at most that of $g$. Let $\xi: 2^{[k-3]} \to \mathbb{R}_{\geq 0}$ be arbitrary for now. For $I \subseteq [k-2]$, we define
\begin{equation*}
\arraycolsep=1pt
\begin{array}{llll}
f(I) &= &g(I) &+ \mathbf{1}_{k-2 \in I} \cdot \xi(I \setminus \{k-2\}) ,\\
f(I \cup \{k-1\}) &= \ \frac{x_{k-1}}{x_{k-1}+x_k} &g(I \cup \{k-1\}) &- \mathbf{1}_{k-2 \in I} \cdot \xi(I \setminus \{k-2\}) ,\\
f(I \cup \{k\}) &= \ \frac{x_k}{x_{k-1}+x_k} &g(I \cup \{k-1\}) &- \mathbf{1}_{k-2 \in I} \cdot \xi(I \setminus \{k-2\}) ,\\
f(I \cup \{k-1, k\})\  &= &0 &+ \mathbf{1}_{k-2 \in I} \cdot \xi(I \setminus \{k-2\}) .
\end{array}
\end{equation*}
Observe that for $\xi = 0$, this definition of $f$ splits up $g$ among $x_{k-1}$ and $x_k$ proportionally. We can interpret the role of $\xi$ as moving equal weights from $I \cup \{k-1\}$ and $I \cup \{k\}$ to their union and intersection; the hope is that for a suitable choice of $\xi$, the weight moved to $I \cup \{k-1,k\}$ will satisfy the constraint $\{i,j\} = \{k-1,k\}$ of~\eqref{LP:decomp}. We claim that, for any choice of $\xi$, $f$ satisfies the all constraints of~\eqref{LP:decomp} except possibly the nonnegativity constraints and the constraint corresponding to $\{i,j\} = \{k-1,k\}$. Observe that for every $I \subseteq [k-2]$, we have
\[ f(I) + f(I \cup \{k-1\}) + f(I \cup \{k\}) + f(I \cup \{k-1,k\}) = g(I) + g(I \cup \{k-1\}) .\]
Let $i,j \in [k]$ with $i<j$. If $j \leq k-2$, then
\begin{align*}
     \sum_{i,j \in I \subseteq [k]} f(I) &= \sum_{i,j \in I \subseteq [k-2]}  f(I) + f(I \cup \{k-1\}) + f(I \cup \{k\}) + f(I \cup \{k-1,k\}) \\
     &= \sum_{i,j \in I \subseteq [k-2]}  g(I) + g(I \cup \{k-1\})  = \sum_{i,j \in I \subseteq [k-1]} g(I) = x_i x_j .
\end{align*}
Similarly, if $j = k-1$, then
\begin{align*}
    \sum_{i,j \in I \subseteq [k]} f(I) &= \sum_{i \in I \subseteq [k-2]}  f(I \cup \{k-1\}) + f(I \cup \{k-1,k\})  
    = \sum_{i \in I \subseteq [k-2]} \frac{x_{k-1}\cdot g(I \cup \{k-1\})}{x_{k-1}+x_k} \\
    &= \frac{x_{k-1}}{x_{k-1}+x_k} \sum_{i,j \in I \subseteq [k-1]} g(I) = \frac{x_{k-1}}{x_{k-1}+x_k} x_i(x_{k-1}+x_k) = x_i x_{k-1} .
\end{align*}
Finally, if $j = k$ and $i \leq k-2$, then
\begin{align*}
    \sum_{i,j \in I \subseteq [k]} f(I) &= \sum_{i \in I \subseteq [k-2]}  f(I \cup \{k\}) + f(I \cup \{k-1,k\}) 
    = \sum_{i \in I \subseteq [k-2]} \frac{x_k\cdot  g(I \cup \{k-1\}) }{x_{k-1}+x_k}\\
    &= \frac{x_k}{x_{k-1}+x_k} \sum_{i,j \in I \subseteq [k-1]} g(I) = \frac{x_k}{x_{k-1}+x_k} x_i(x_{k-1}+x_k) = x_i x_k .
\end{align*}
The missing constraint, $\{i,j\} = \{k-1,k\}$, is
\[ \sum_{i,j \in I \subseteq [k]} f(I) = \sum_{I \subseteq [k-2]} f(I \cup \{k-1,k\}) = \sum_{I \subseteq [k-3]} \xi(I) = x_{k-1} x_k .\]
Choosing $\xi(I) = \frac{x_{k-1}x_k}{x_{k-2}(x_{k-1}+x_k)} g(I \cup \{k-2, k-1\})$ satisfies this constraint:
\begin{align*}
\sum_{I \subseteq [k-3]} \xi(I) &= \frac{x_{k-1}x_k}{x_{k-2}(x_{k-1}+x_k)} \sum_{k-2,k-1 \in I \subseteq [k-1]} g(I) \\
&= \frac{x_{k-1}x_k}{x_{k-2}(x_{k-1}+x_k)} x_{k-2}(x_{k-1}+x_k) = x_{k-1} x_k .
\end{align*}
The nonnegativity of $f$ either follows directly from the nonnegativity of $g$, or from the following calculation for this choice of $\xi$: for $I \subseteq [k-2]$ with $k-2 \in I$, we have
\begin{align*}
    f(I \cup \{k-1\}) &= \frac{x_{k-1}(x_{k-2}-x_k)}{x_{k-2}(x_{k-1}+x_k)} g(I \cup \{k-1\}) \geq 0 ,\\
    f(I \cup \{k\}) &= \frac{x_k(x_{k-2}-x_{k-1})}{x_{k-2}(x_{k-1}+x_k)} g(I \cup \{k-1\}) \geq 0 .
\end{align*}
Thus, with this choice of $\xi$, $f$ is feasible for~\eqref{LP:decomp}. We now compute the $\mathbf{c}$-cost of $f$ as
\begin{align*}
    \norm{f}_\mathbf{c} =& \sum_{I \subseteq [k]} c_{|I|} f(I) = \sum_{I \subseteq [k-3]} \sum_{J \subseteq \{k-2,k-1,k\}} c_{|I\cup J|} f(I\cup J)\\
    =& \sum_{I \subseteq [k-3]} \bigg( c_{|I|} g(I) + c_{|I|+1} \left( 
g(I \cup \{k-2\}) + g(I \cup \{k-1\}) + \xi(I) \right)\\
&\ \ \ \ \ \ \ \ \ \ \ + c_{|I|+2} \left( g(I \cup \{k-2,k-1\}) - 2\xi(I) \right) + c_{|I|+3} \xi(I) \bigg)\\
   =& \sum_{I \subseteq [k-1]} c_{|I|} g(I) + \sum_{I \subseteq [k-3]} \left( c_{|I|+1} - 2c_{|I|+2} + c_{|I|+3} \right) \xi(I) \leq \sum_{I \subseteq [k-1]} c_{|I|} g(I) = \norm{g}_\mathbf{c} ,
\end{align*}
since $c_{|I|+1} - 2c_{|I|+2} + c_{|I|+3} \leq 0$ and $\xi(I) \geq 0$.

By induction, we may take $\norm{g}_\mathbf{c}$ to be at most
\[ c_2 \max_{J \subseteq [k-1]} \left( \sum_{j \in J} x_j' \right) \left( \sum_{j \not\in J} x_j' \right) \leq c_2 \max_{J \subseteq [k]} \left( \sum_{j \in J} x_j \right) \left( \sum_{j \not\in J} x_j \right) ,\]
where $x_j' = x_j$ if $j \in [k-2]$ and $x_{k-1}' = x_{k-1} + x_k$.
\end{proof}

We quickly derive Theorem~\ref{thm:solve_LP} from Lemmas~\ref{lem:LP_reduction}, \ref{lem:coverLP}, and~\ref{lem:decompLP}.

\begin{proof}[Proof of Theorem~\ref{thm:solve_LP}]
Let $H$ be a complete multipartite graph with parts $V_1, \dots, V_k$, and let $x_i = |V_i|/n$ for $i \in [k]$, where $x_1 \geq \ldots \geq x_k$.

For the first part of Theorem~\ref{thm:solve_LP}, by Lemmas~\ref{lem:LP_reduction} and~\ref{lem:coverLP}, we have
\[ \coveroptfrac{t}{H,\mathbf{c}} \leq n^t c_t \left( \prod_{i=1}^{t-1} x_i \right) \left( 1 - \sum_{i=1}^{t-1} x_i \right) = c_t \left( \prod_{i=1}^{t-1} |V_i| \right) \left( \sum_{i=t}^k |V_i| \right) .\]
By standard optimization techniques, the right side is maximized when $|V_1|$, $|V_2|$, \dots, $|V_{t-1}|$, $|V_t| + |V_{t+1}| + \ldots + |V_k|$ are as equal as possible. This optimum value is simply $\coveroptfrac{t}{\turangraph{n}{t},\mathbf{c}}$.

For the second part of Theorem~\ref{thm:solve_LP}, by Lemmas~\ref{lem:LP_reduction} and~\ref{lem:decompLP}, we have
\[ \decompositionoptfrac{2}{H,\mathbf{c}} \leq n^2 c_2 \max_{J \subseteq [k]} \left( \sum_{j \in J} x_j \right) \left( \sum_{j \not\in J} x_j \right) = c_2 \max_{J \subseteq [k]} \left( \sum_{j \in J} |V_j| \right) \left( \sum_{j \not\in J} |V_j| \right) .\]
By standard optimization techniques, the right side is maximized when there is a bipartition of the $V_i$'s that is as equal as possible, matching the $\decompositionoptfrac{2}{\turangraph{n}{2}, \mathbf{c}}$.
\end{proof}

We take the rest of this section to discuss optimal solutions to~\eqref{LP:cover} and~\eqref{LP:decomp}. We encode the variables of the dual LP as a function $g : \binom{[k]}{t} \to \mathbb{R}$. The dual LP to~\eqref{LP:cover} is
\begin{equation}\label{LP:coverdual}\tag{CCdual}
\begin{array}{rll}
\text{maximize} & \displaystyle\sum_{T \in \binom{[k]}{t}} g(T) \displaystyle\prod_{i \in T} x_i & \\
\text{subject to} & \displaystyle\sum_{T \subseteq I} g(T) \leq c_{|I|}, &\forall I \subseteq [k], \\
& g(T) \geq 0, &\forall T \in \binom{[k]}{t}.
\end{array}
\end{equation}
The dual LP to~\eqref{LP:decomp} is the same, just without the nonnegativity constraints:
\begin{equation}\label{LP:decompdual}\tag{CDdual}
\begin{array}{rll}
\text{maximize} & \displaystyle\sum_{T \in \binom{[k]}{t}} g(T) \displaystyle\prod_{i \in T} x_i & \\
\text{subject to} & \displaystyle\sum_{T \subseteq I} g(T) \leq c_{|I|}, &\forall I \subseteq [k].
\end{array}
\end{equation}
We remark on two interesting cases: when $c_i = 1$ for all $i$, and when $c_i = i-1$ for all $i$. When $c_i = 1$ for all $i$, we claim that the optimal value of~\eqref{LP:cover} is $\prod_{i=1}^t x_i$. That is, we can determine $\coveroptfrac{t}{H}$ for all complete multipartite graphs. This optimal value is realized by $f([k]) = \prod_{i=1}^t x_i$ and $0$ otherwise, where $x_1 \geq \ldots \geq x_k$. Optimality is witnessed by $g([t]) = 1$ and $0$ otherwise.

Now consider~\eqref{LP:decomp} when the cost vector is $c_i = i-1$ and $t=2$, that is, the original setting of Erd\H{o}s's problem. These weights may seem arbitrary, but they turn out to be special: the dual~\eqref{LP:coverdual} to the cover problem is a maximum spanning tree LP, that is, interpreting $g$ as (fractionally) selecting the edges of a weighted complete graph, the constraints are that every $i$ vertices induces at most $i-1$ edges. With this interpretation in hand, the optimal solution to~\eqref{LP:coverdual} is just $g(\{1,j\}) = 1$ for all $2 \leq j \leq k$, and $0$ otherwise, with optimal value $x_1(1-x_1)$. The dual~\eqref{LP:decompdual} to the decomposition problem drops the nonnegativity constraints of~\eqref{LP:coverdual}, which make many techniques applicable to~\eqref{LP:coverdual} no longer valid. Nevertheless, we can still use the same $g$ to show that the optimal value of~\eqref{LP:decomp} is at least $x_1(1-x_1)$. We believe that this should in fact be the optimal value of~\eqref{LP:decomp}, but we were unable to prove it. It appears that the optimal solutions depend on various relations among the $x_i$, such as whether $x_1 - x_2 \geq x_2 - x_3$.

\section{Integer to fractional: proof of Theorem~\ref{thm:asymptotic_bound}}\label{sec:asymptotic}

The main tools of the proof of Theorem~\ref{thm:asymptotic_bound} are Szemer\'edi's regularity lemma and the Frankl--R\"odl nibble. The following is a high-level sketch of the proof. Given $\ep>0$ and an $n$-vertex graph $G$, we `clean' $G$ in a standard way using the regularity lemma to obtain a subgraph $\tilde{G}$ of $G$ such that $\coveropt{t}{G} \leq \coveropt{t}{\tilde{G}} + \ep n^t$. This $\tilde{G}$ has a nice partite structure, and we also need $\tilde{G}$ to have bounded clique number for technical reasons (for example, the dependence of $\ep_0$ on $k$ in Lemma \ref{lem:regularity_counting}; see the remark at the end of Section~\ref{sec:asymptotic:frac_to_int}). We use the information contained in an optimal fractional $t$-clique cover of $\tilde{G}$ to partition $\tcliques{t}{\tilde{G}}$ among  $K \in \cliques{\tilde{G}}$, and then for each $K \in \cliques{\tilde{G}}$, we apply the Frankl--R\"odl nibble to obtain an integer $t$-clique cover (of the required total weight) of the associated $t$-cliques. This shows that for these `cleaned' graphs, the fractional and integer $t$-clique cover numbers are close to each other. Since the fractional $t$-clique cover number is bounded by Theorem~\ref{thm:frac}, we get the desired bound on $\coveropt{t}{\tilde{G}}$ and hence on $\coveropt{t}{G}$ as well.

In Section~\ref{sec:asymptotic:prelim} we describe  Szemer\'edi's regularity lemma and the associated cleaning. In Section~\ref{sec:asymptotic:nibble} we state the Frankl--R\"odl nibble method and give our  application of it. In Section~\ref{sec:asymptotic:frac_to_int} we describe  how to use a fractional $t$-clique cover of a cleaned graph to obtain an integer $t$-clique cover of the same graph, via an application of  the Frankl--R\"odl nibble. Finally, in Section~\ref{sec:asymptotic:final_proof} we give an additional technical lemma needed to get bounded clique number, and then we compile all the pieces into a proof of Theorem~\ref{thm:asymptotic_bound}.

One bit of notation is particularly helpful in simplifying the technicality of our statements and proofs. We say $x = a \pm b$ (where $b$ is always positive) to mean $a-b \leq x \leq a+b$. Furthermore, for a monotone function $f$ on $\mathbb{R}_{\geq 0}$, we say $x = f(a \pm b)$ to mean $f(a-b) \leq x \leq f(a+b)$. We read this notation going left-to-right, so $a \pm b = c \pm d$ means $[a-b,a+b] \subseteq [c-d,c+d]$. Finally, we have $(a \pm b) + (c \pm d) = (a+c) \pm (b+d)$; and if $a$ and $c$ are positive, $(a \pm b)(c \pm d) = (ac + bd) \pm (ad + bc)$.

\subsection{Regularity Lemma preliminaries}\label{sec:asymptotic:prelim}

For two disjoint sets $U, W$ of vertices in a graph $G$, we denote by $e(U,W)$ the number of edges of $G$ with one endpoint in $U$ and the other in $W$. Let $d(U, W)$ be the \emph{density} of the pair $(U, W)$ defined by 
\[ d(U,W) := \frac{e(U, W)}{|U||W|}. \]
For $\ep > 0$, we say the pair $(U,W)$ is \emph{$\ep$-regular} in $G$ if $U$ and $W$ are disjoint and for every $U' \subseteq U$ and $W' \subseteq W$ such that $|U'| \geq \ep |U|$ and $|W'| \geq \ep |W|$ we have 
\[ |d(U', W') - d(U, W)| < \ep.\]

We shall use the following ``equitable'' version of the Regularity Lemma~\cite{Sz}.
\begin{theorem}[Regularity Lemma]\label{thm:regular}
For every $\ep > 0$ and positive integer $m_0$, there exists a constant $M = M(m_0, \ep)$ such that for every graph $G$, there exists a partition $P = \{V_0, V_1, \dots, V_\ell\}$ of $V(G)$ with $m_0 < \ell \leq M$ such that the following hold: 
\begin{itemize}
    \item[(i)] $|V_0| < \ep |V(G)|$ and $|V_1| = \ldots = |V_\ell|$;
    \item[(ii)] all but at most $\ep \binom{\ell}{2}$ pairs $(V_i, V_j)$ with $1 \leq i < j \leq \ell$ are $\ep$-regular. 
\end{itemize}
\end{theorem}

For a graph $G$, we say a partition $P = \{V_0, V_1, \dots, V_\ell\}$ of $V(G)$ satisfying $(i)$ and $(ii)$ in  Theorem \ref{thm:regular} is an $\ep$-{\it regular partition}. We use this regular partition as a guide to clean the graph.
The proof of the lemma below is standard and we omit it.

\begin{lemma}[Cleaning]\label{lem:cleaning}
For every $\ep > 0$ and positive integer $m_0$, there exists a constant $M = M(\ep, m_0) >0$ such that the following holds. Let $G$ be a graph, and let $d>0$. Then $G$ has an $n$-vertex subgraph $\tilde{G}$, where  $n = |V(\tilde{G})| \geq (1-\ep)|V(G)|$, with the following properties:
\begin{itemize}
    \item[(i)] $V(\tilde{G})$ has a partition $P = \{V_1,\dots, V_\ell\}$, where each $V_i$ is an independent set, $m_0 < \ell \leq M$, $|V_1| = \ldots = |V_\ell| = n/\ell$;
    \item[(ii)] $V_i \cup V_j$ is an independent set if $(V_i, V_j)$ is not an $\ep$-regular pair with density at least $d$;
    \item[(iii)] $|E(G)| \leq |E(\tilde{G})| + (2\ep + 1/m_0 + d) n^2$.
\end{itemize}
\end{lemma}

Adding back  the edges of $G$ to the cleaned graph $\tilde{G}$ does not increase the integer $t$-clique cover too much; see Lemma~\ref{lem:remove_edges_cover}. Henceforth, we will be working with the cleaned graph $\tilde{G}$ from Lemma~\ref{lem:cleaning}. 

Call $G$ an \emph{$(n,\ell,\ep,d)$-graph} if it has $n$ vertices which can be partitioned into independent sets $V_1, \dots, V_\ell$ such that for every $1 \leq i < j \leq \ell$, we have $|V_i| = |V_j|$, and either $(V_i, V_j)$ is an $\ep$-regular pair with density at least $d$ or $V_i \cup V_j$ is an independent set. For an $(n,\ell,\ep,d)$-graph $G$ witnessed by the partition $P = \{V_1, \dots, V_\ell\}$, let $R = R(G,n,\ell,\ep,d,P)$ be the graph with $\ell$ vertices $v_1, \dots, v_\ell$, where $v_i v_j$ is an edge of $R$ if and only if $(V_i, V_j)$ has density at least $d$ in $G$. We call $R$ the \emph{reduced graph}. For $I\subseteq[\ell]$ and $K = \{v_i\}_{i \in I} \in \K(R)$, let
\[ G(K) = \left\{ \{w_i\}_{i \in I} \in \cliques{G} : w_i \in V_i \text{ for every } i \in I \right\} ,\]
that is, $G(K)$ is the set of cliques of $G$ `corresponding' to $K$. 

We need the following simple generalization of the counting lemma: for every $T \in \tcliques{t}{R}$ and $K \in \cliques{R}$ with $T \subseteq K$, almost all $T' \in G(T)$ are contained in approximately the same number of $K' \in G(K)$, as is expected given the densities of the $\ep$-regular pairs $(V_i, V_j)$ in $G$. Before we claim this in Lemma~\ref{lem:regularity_counting}, we need a standard lemma about subsets of regular pairs.

\begin{lemma}\label{lem:reg_subpair}
Let $(V_1, V_2)$ be an $\ep$-regular pair of density $d$ with $\ep \leq 1/4$, and let $V_i' \subseteq V_i$ be such that $|V_i'| \geq \sqrt{\ep} |V_i|$ for $i \in [2]$. Then $(V_1', V_2')$ is a $\sqrt{\ep}$-regular pair of density $d \pm \ep$.
\end{lemma}

\begin{proof}
As $(V_1,V_2)$ is an $\ep$-regular pair, and $|V_i'| \ge \sqrt{\ep}|V_i|$ for every $i \in [2]$, we have $|d(V_1',V_2') - d(V_1,V_2)| < \ep$, which implies $d(V_1',V_2') = d \pm \ep$. For every $i \in [2]$ and for every $A_i \subseteq V_i'$ such that $|A_i| \geq \sqrt{\ep} |V_i'| \geq \ep |V_i|$, since $(V_1, V_2)$ is $\ep$-regular, we have
\[ |d(A_1,A_2) - d(V_1',V_2')| \leq |d(A_1,A_2) - d(V_1,V_2)| + |d(V_1,V_2) - d(V_1',V_2')| < 2\ep \leq \sqrt{\ep} ,\]
showing that $(V_1',V_2')$ is a $\sqrt{\ep}$-regular pair. 
\end{proof}

The first part of Lemma \ref{lem:regularity_counting} follows from  a standard application of Szemer\'{e}di's Regularity Lemma, which has appeared in the literature in various forms, see for example \cite[Theorem~14]{KSSS}. For the second part of Lemma \ref{lem:regularity_counting}, similar results have been known in the special case when $t = 2$, see \cite[Lemma 15]{HR} or \cite[Lemma 2.2]{Y}.

\begin{lemma}\label{lem:regularity_counting}
For every integer $k > 1$ and pair of positive real numbers $d$ and $\gamma$, there exists an $\ep_0 = \ep_0(k,d,\gamma) > 0$ such that for $\ep < \ep_0$, the following holds. Let $G$ be a graph with vertex partition $\{V_1, \dots, V_k\}$ such that each $V_i$ is an independent set and $(V_i,V_j)$ is an $\ep$-regular pair of density $d_{i,j} \geq d$ for all $1 \leq i < j \leq k$.  Then the following holds (interpreting empty products as $1$):
\begin{enumerate}[(i)]
    \item  We have that \[ |\tcliques{k}{G}| = (1 \pm \gamma) \prod_{i=1}^{k} \left( |V_i| \prod_{j=i+1}^{k} d_{i,j} \right) .\]
    \item For every $1 \leq t < k$ there exist at least
\[ (1-\gamma) \prod_{i=1}^t \left( |V_i| \prod_{j=i+1}^{t} d_{i,j} \right) \]
$t$-cliques in $G[\bigcup_{i=1}^t V_i]$, such that each  is  contained in
\[ (1 \pm \gamma)\left(  \prod_{i=t+1}^{k} |V_i| \right) \left( \prod_{\{i,j\} \in I(t,k)} d_{i,j} \right) \]
of the $k$-cliques in $G$, where $I(x,y) := \{ \{i,j\} : 1 \le i < y \text{ and }  \max\{i,x\} < j \le y\}$.
\end{enumerate}
\end{lemma}

\begin{proof}
We proceed by induction on $k$. Our base case is $k=2$, when $t=1$. By definition, we have $|\tcliques{2}{G}| = e(V_1,V_2) = d_{1,2}\cdot |V_1||V_2|$, establishing (i).  We set $\ep_0(2,d,\gamma)  =\min\{ \gamma/2, \gamma d\}$. For  $\ep < \ep_0(2,d,\gamma)$ we let $V_1'$ be the set of vertices $v \in V_1$ such that $|N(v) \cap V_2| < (d_{1,2}-\ep)|V_2|$. Since $d(V_1',V_2) < d_{1,2}-\ep$ and $(V_1, V_2)$ is $\ep$-regular, we have that $|V_1'| < \ep |V_1|$. We reach a similar conclusion regarding the number of $v \in V_1$ such that $|N(v) \cap V_2| > (d_{1,2}+\ep)|V_2|$, completing the proof of (ii) of the base case.

For the induction step, assume $k>2$. Define inductively
\[ \ep_0(k,d,\gamma) = \min\left\{ 1/4, \gamma d^{\binom{k}{2}}/ 9 (k-1)^2, \ep_0(k-1,d/2,\gamma/2)^2 \right\} ,\]
so that for all $0 < \ep < \ep_0$, for all $I \subseteq \binom{[k]\setminus\{1\}}{2}$, and for all $d^*_{i,j}$ for $2 \leq i < j \leq k$ with $d \leq d^*_{i,j} \leq 1$, we have
\begin{equation}\label{eq:regularity_counting_1}
    \prod_{\{i,j\} \in I} \left( d^*_{i,j} \pm \ep \right) = (1 \pm \ep/d)^{|I|} \prod_{\{i,j\} \in I} d^*_{i,j} = \left( 1 \pm \gamma/9 \right) \prod_{\{i,j\} \in I} d^*_{i,j} .
\end{equation}

Let $\ep < \ep_0(k,d,\gamma)$.  As before, at least $(1-2(k-1)\ep) |V_1|$ vertices $v \in V_1$ satisfy $|N(v) \cap V_i| = (d_{1,i} \pm \ep)|V_i|$ for all $2 \leq i \leq k$. Fix such a $v \in V_1$, and let $V_i' = N(v) \cap V_i$. As $\ep < 1/4$, by Lemma~\ref{lem:reg_subpair}, the pairs $(V_i',V_j')$ for $2 \leq i < j \leq k$ are all $\sqrt{\ep}$-regular with density $d'_{i,j} := d(V_i',V_j') = d_{i,j} \pm \ep$. By our choice of $\ep$, $d'_{i,j} \geq d/2$ for all $2 \leq i < j \leq k$. As $\ep_0(k,d,\gamma) < \ep_0(k-1,d/2,\gamma/2)^2$,  by induction with $\gamma$ as $\gamma/2$, there are
\begin{align*}
    &(1 \pm \gamma/2) \prod_{i=2}^k \left( |V_i'| \prod_{j=i+1}^{k} d'_{i,j} \right) = (1 \pm \gamma/2) \prod_{i=2}^k \left( |V_i'| \prod_{j=i+1}^{k} (d_{i,j} \pm \ep) \right)\\
    = &(1 \pm \gamma/2) \left( \prod_{i = 2}^k  |V_i| \right)  \left( \prod_{1 \leq i < j \leq k} (d_{i,j} \pm \ep) \right)
    =  (1 \pm 2\gamma/3) \left( \prod_{i = 2}^k  |V_i| \right)  \left( \prod_{1 \leq i < j \leq k} d_{i,j} \right)
\end{align*} 
$k$-cliques in $G$ containing $v$, where the last (in)equality follows from \eqref{eq:regularity_counting_1}. This proves the second part of the lemma when $t = 1$. Summing over all $v \in V_1$, we get
\begin{align*}
    |\tcliques{k}{G}| &= (1 \pm 2(k-1)\ep) |V_1| (1 \pm 2\gamma/3)  \left( \prod_{i = 2}^k  |V_i| \right)  \left( \prod_{1 \leq i < j \leq k} d_{i,j} \right) \pm 2(k-1)\ep \prod_{j=1}^k |V_j|\\
    &= (1 \pm \gamma) \prod_{j=1}^k \left( |V_j| \prod_{i=1}^{j-1} d_{i,j} \right), 
\end{align*}
where we use that $\ep \leq \gamma d^{\binom{k}{2}} / 9 (k-1)^2$. This concludes the first part of the lemma.

\noindent For the $t>1$ statement, we see that by induction, at least
\begin{align*}
    (1-\gamma/2) \prod_{i=2}^t \left( |V_i'| \prod_{j=i+1}^{t} d'_{i,j} \right) &\geq (1-\gamma/2) \left( \prod_{i=2}^t  |V_i| \right) \left( \prod_{1 \leq i < j \leq t} (d_{i,j}-\ep) \right)\\
    &\geq (1-2\gamma/3) \left( \prod_{i=2}^t  |V_i| \right) \left( \prod_{1 \leq i < j \leq t} d_{i,j} \right) 
\end{align*}
of the $(t-1)$-cliques in $G[\bigcup_{i=2}^t V_i']$ (last inequality follows from \eqref{eq:regularity_counting_1}) are contained in 
\begin{align*}
    &(1\pm \gamma/2) \left( \prod_{i=t+1}^k |V_i'|\right) \left( \prod_{i = t+1}^{k}\prod_{j=i+1}^{k} d'_{i,j}\right) \left(\prod_{i = 2}^{t}\prod_{j = t+1}^{k} d'_{i,j}\right) \\  =& (1 \pm \gamma/2)\left( \prod_{i=t+1}^k |V_i|\right)\left(\prod_{(i,j) \in I(t,k)} \left( d_{i,j} \pm \ep \right) \right)
    = (1 \pm \gamma)\left(  \prod_{i=t+1}^{k} |V_i| \right) \left( \prod_{(i,j) \in I(t,k)} d_{i,j} \right)
\end{align*}
$(k-1)$-cliques in $G[\bigcup_{i=2}^k V_i']$, where the last equality follows from \eqref{eq:regularity_counting_1}. Since this is true for at least $(1-2(k-1)\ep) |V_1|$ of the $v \in V_1$, we have that at least
\[ (1-2\gamma/3)(1-2(k-1)\ep) \left( \prod_{i=1}^t  |V_i| \right) \left( \prod_{(i,j) \in I(1,t)} d_{i,j} \right) \geq (1-\gamma) \prod_{j=1}^t \left( |V_j| \prod_{i=1}^{j-1} d_{i,j} \right) \]
of the $t$-cliques in $G[\bigcup_{i=1}^t V_i]$ are contained in
\[ (1 \pm \gamma)\left(  \prod_{i=t+1}^{k} |V_i| \right) \left( \prod_{(i,j) \in I(t,k)} d_{i,j} \right) \]
$k$-cliques in $G$.
\end{proof}

\subsection{The Nibble Method}\label{sec:asymptotic:nibble}

For a hypergraph $\mathcal{H}$, we denote by $d_\mathcal{H}(v)$ the \emph{degree} of a vertex $v$, which is the number of edges of $\mathcal{H}$ containing $v$. The \emph{co-degree} of a pair of vertices $u,v$, denoted by $d_\mathcal{H}(u,v)$, is the number of edges of $\mathcal{H}$ containing both $u$ and $v$. The following simple statement of the Frankl--R\"odl nibble, due to Pippenger (see \cite[Theorem 4.7.1]{AS} or \cite[Theorem 8.4]{F}), states that a nearly regular $r$-uniform $n$-vertex hypergraph with small co-degrees has a nearly perfect edge cover/packing, that is, a cover/packing with nearly $n/r$ edges.

\begin{theorem}\label{thm:nibble}
For every integer $r \ge 2$ and reals $C \ge 1$ and $\delta' > 0$, there exist $\gamma' = \gamma'(r,C,\delta') > 0$ and $D_0 = D_0(r,C,\delta')$ such that for every $n$ and $D \ge D_0$ the following holds. Let $\mathcal{H}$ be an $n$-vertex $r$-uniform hypergraph which satisfies the following conditions:
\begin{enumerate}[(i)]
    \item for all but at most $\gamma' n$ vertices $x \in V({\cal H})$, $d_\mathcal{H}(x) = (1 \pm \gamma')D$,
    \item for all $x \in V(\mathcal{H})$, $d_\mathcal{H}(x) < CD$, and
    \item for any two distinct $x,y \in V(\mathcal{H})$, $d_\mathcal{H}(x,y) < \gamma' D$.
\end{enumerate}
Then $\mathcal{H}$ has an edge packing (i.e., matching) of size at least $(1 - \delta')(n/r)$.
\end{theorem}

We remark here that Theorem \ref{thm:nibble} is stated in \cite[Theorem 4.7.1]{AS} in terms of the existence of nearly perfect coverings. It is well-known that the existence of a nearly perfect covering implies the existence of a nearly perfect packing.

We will apply Theorem~\ref{thm:nibble} in the following way. Recall that $R$ is the reduced graph of the cleaned graph $G$ (see Section~\ref{sec:asymptotic:prelim} for the cleaning). Fix $K \in \cliques{R}$, and recall that $G(K)$ is the set of cliques in $G$ `corresponding' to $K$. We want a subset of $G(K)$ which covers as efficiently as possible all  $t$-cliques corresponding to some $t$-clique $T \subseteq K$. To do this, we set up a $\binom{k}{t}$-uniform $\binom{k}{t}$-partite hypergraph $\mathcal{H}_K$ whose parts are $G(T)$ for $t$-cliques $T \subseteq K$ and whose edges are given by $G(K)$, where a hyperedge $\tilde{K} \in G(K)$ has vertices $\tilde{T} \in G(T)$
with $\tilde{T} \subset \tilde{K}$.

Unfortunately, $\mathcal{H}_K$ cannot have a nearly optimal edge cover, since the parts of $\mathcal{H}_K$ do not necessarily have the same size --- their sizes depend on the densities of the pairs of parts of $G$, as given by Lemma~\ref{lem:regularity_counting}. However, when we randomly sample subsets of these parts so that these subsets all have about the same size, then we can use Lemma~\ref{lem:regularity_counting} to check the hypotheses of Theorem~\ref{thm:nibble} and thus obtain a nearly perfect edge cover. How exactly we choose these random subsets depends on the fractional $t$-clique cover/decomposition of $G$ that we start with. We postpone these details until Section~\ref{sec:asymptotic:frac_to_int}. For now, we show in general terms that the nibble method can be applied to this `randomly thinned' subgraph of $\mathcal{H}_K$.

\begin{lemma}\label{lem:random_nibble}
For every pair of integers $k$ and $t$ such that $1 \le t < k$ and pair of positive real numbers $d \le 1$ and $\delta$, there exists $\ep_0 = \ep_0(k,d,\delta) > 0$ such that for all $\ep < \ep_0$ and for all $\alpha_0 > 0$, there exists $n_0 = n_0(k,\alpha_0,d,\delta)$ such that the following holds.
Let $G$ be a graph with vertex partition $\{V_1, \dots, V_k\}$ where each $V_i$ is an independent set of size $n \geq n_0$ and $(V_i, V_j)$ is an $\ep$-regular pair of density $d_{i,j} \geq d$ for all $1 \leq i < j \leq k$, and let $\alpha_0 \leq \alpha \le d^{\binom{t}{2}}$.

For each $T \in \binom{[k]}{t}$, let $G(T) = \{ T' \in \tcliques{t}{G} : \forall i \in T, |T' \cap V_i| = 1 \}$, let $p_T = \frac{\alpha n^t}{|G(T)|}$, and let $G(T)^*$ be a random subset of $G(T)$ where we include each element independently with probability $p_T$, independently of the other $T' \in \binom{[k]}{t}$. Then with probability at least $1-1/n$,
there exists a partition of $\bigcup_{T \in \binom{[k]}{t}} G(T)^*$ into two sets $\mathcal{S}_1$ and $\mathcal{S}_2$
such that 
\begin{enumerate}[(i)]
    \item $\mathcal{S}_1$ can be partitioned into sets of the form $\{T' : T' \subseteq K'\}$ for $K' \in \tcliques{k}{G}$;
    \item $|\mathcal{S}_1| = (1 \pm \delta)\binom{k}{t}\alpha n^t$;
    \item $|\mathcal{S}_2| \le \delta \alpha n^t$.
\end{enumerate}
\end{lemma}

Note that, crucially, $\alpha_0$ in Lemma~\ref{lem:random_nibble} may be taken arbitrarily close to $0$, bounded away from zero independently of only $n$. This is important, because later $\alpha$ may depend on these regularity lemma parameters in a way we cannot meaningfully control (that is, $\alpha$ will be inversely polynomial in $\ell$, the number of parts in the regularity partition, which may be an exponential tower of height inversely polynomial in $\ep$; see the remark at the end of Section~\ref{sec:asymptotic:frac_to_int}). Also note that $1-1/n$ is not the optimal probability --- what we really get is an error probability exponentially small in $n$. 

When $t=2$, we recover the main technical lemma of Yuster \cite{Y} in the special case of clique packings. Conceptually, the $t=2$ case may be simpler than the $t>2$ case, as the randomly thinned $\mathcal{H}_K$ is just a random spanning subgraph of our $\ep$-regular partition, and so an off-the-shelf counting lemma could be applied. We note that there are dependencies among the indicator random variables whose sum is the degree of a given vertex in the randomly thinned $\mathcal{H}_K$. In order to estimate these degrees, Yuster splits each of these random variables into classes of independent variables and applies the Chernoff bound to these classes. We instead apply McDiarmid's concentration inequality~\cite{MD} for Lipschitz functions of independent random variables. We give these inequalities now.

\begin{proposition}[Chernoff bound]\label{prop:chernoff}
Let $X$ be the sum of finitely many independent indicator random variables. Then for every $0 \leq \eta < 1$, we have
\[ \mathbb{P}\left( |X-\mathbb{E}[X]| \geq \eta \mathbb{E}[X]\right) \leq 2 \exp\left( - \frac{\eta^2}{3} \mathbb{E}[X]  \right). \]
\end{proposition}

\begin{theorem}\label{thm:McDiarmid}
Let $f : \cX_1 \times \dots \times \cX_n \to \R$ and let $X_1,\ldots,X_n$ be independent random variables such that $X_i \in \cX_i$ for every $i \in [n]$. If for all $i \in [n]$ and for all $x_1 \in \cX_1, \dots, x_n \in \cX_n$, the function $f$ satisfies 
\begin{equation*}
   \sup_{x'_i \in \cX_i} \left|f(x_1,\dots, x_{i-1},x_i, x_{i+1}, \dots, x_n) - f(x_1,\dots, x_{i-1},x_i', x_{i+1}, \dots, x_n)\right| \leq b_i  
\end{equation*}
for some $b_1, \dots, b_n \in \R_{\geq 0}$, then for every $\mu > 0$,
\begin{align*}
\pr\left(\left|f(X_1, \dots, X_n) - \E[ f(X_1, \dots, X_n)]\right| \geq \mu \right) \leq 2 \exp\left(- \dfrac{2\mu^2}{\sum_{i=1}^n b_i^2}\right).    
\end{align*}
\end{theorem}

\begin{proof}[Proof of Lemma \ref{lem:random_nibble}]
Throughout the proof, we will show that several events occur with high probability. Whenever this happens, we will tacitly condition on these events occurring, omitting this conditioning from our notation.

Let $1\le t< k$ be integers, and let $\alpha_0, d, \delta > 0$ be given. Let $\alpha_0 \leq \alpha \le d^{\binom{t}{2}}$. Suppose
\begin{equation}\label{eq:C_and_delta_selection}
    C := C(t,k,d) = 2^{\binom{k}{t}} d^{-2\binom{k}{2}\binom{k-2}{t-2}}\quad\quad  \text{ and } \quad\quad \delta' < \frac{\delta}{2} \binom{k}{t}^{-1}. 
\end{equation}
Let $\gamma'({\binom{k}{t}},C,\delta')$ and $D_0({\binom{k}{t}},C,\delta')$ be the required parameters such that every $n'$-vertex, ${\binom{k}{t}}$-uniform hypergraph satisfying  conditions  (i), (ii),  and (iii) of Theorem~\ref{thm:nibble} has a matching of size at least $(1 - \delta')n'{\binom{k}{t}}^{-1}$. We plan to apply Theorem~\ref{thm:nibble} to the $\binom{k}{t}$-partite $\binom{k}{t}$-uniform hypergraph $\mathcal{H}'$ whose parts are the sets  $G(T)^*$ and whose edges are the collections of $t$-cliques which are subcliques of the same $k$-clique. We will define $\mathcal{H}'$ precisely later. Since each $G(T)^*$ has size around $\alpha n^t$ (which will be shown to hold with high probability using the Chernoff bound), Theorem~\ref{thm:nibble} will yield $(1-\delta') \alpha n^t$ pairwise disjoint edges of $\mathcal{H}'$, which will correspond to $(1-\delta') \alpha n^t$ $k$-cliques of $G$, composing (the partition of) $\mathcal{S}_1$.
 
 First, we set up some preliminaries. For $Q \subseteq [k]$, let $q=|Q|$, and let
\[ d_{Q} = \prod_{\{i,j\} \in \binom{Q}{2}} d_{i,j} \quad \quad \text{ and } \quad\quad \hat{d}_Q = \prod_{\{i,j\} \in \binom{Q}{2}} d_{i,j}^{\binom{q-2}{t-2}} = d_{Q}^{\binom{q-2}{t-2}}.\]
By repeated applications of Lemma~\ref{lem:regularity_counting}, we may suppose that for every $\gamma > 0$ there exists $\ep_1(k,d,\gamma)$ such that for every $\ep < \ep_1$ and $Q \subseteq [k]$ of size  $ q$ we have
\begin{equation}\label{eq:sizeG(T)}
    |G(Q)| = (1 \pm \gamma^2) n^q d_{Q}.
\end{equation}

 For  $\tilde{T} \in \binom{[k]}{t}$, we denote by $\deg_G(\tilde{T})$  the number of $k$-cliques in $G$ containing $\tilde{T}$. Furthermore, we say  $\tilde{T}$ is $\gamma$-\textit{good} if  $\deg_{G}(\tilde{T}) = (1 \pm \gamma) n^{k-t} \dfrac{d_{[k]}}{d_T}$.  For  ${T} \in \binom{[k]}{t}$ let $G(T)_\gamma$ be the set of  $\gamma$-good $\tilde{T} \in G(T)$. If $\tilde{T} \in G(T)$ is not $\gamma$-good, then we say it is $\gamma$-\emph{bad}, and we let $G(T)_{\bar{\gamma}} = G(T) \setminus G(T)_\gamma$ denote the set of  $\gamma$-bad cliques. Again, by repeated applications of  Lemma~\ref{lem:regularity_counting}, we may suppose that there exists an $\varepsilon_2(k,d,\gamma)$ such that for all $T \in \binom{[k]}{t}$,  if $\ep < \ep_2$, then 
\begin{equation}\label{eq:size-bad-G(S)}
|G(T)_{\bar{\gamma}}| \leq \gamma^2 n^{t} d_T.
\end{equation} 
From now on, we fix $\gamma = \gamma(k,t,d,\delta,\gamma') < 1/2$ and $\delta' = \delta'(k,t,\delta)$; such that  the following holds:
\begin{equation}\label{eq:gamma_1}
(1+\gamma)\gamma \leq 1 - \gamma, \quad\quad\  \text{ and } \quad\quad (1 \pm \gamma^2)^{-1} = (1 \pm \gamma),
\end{equation}
\begin{equation}\label{eq:gamma_2}
    \gamma/(1 - 2 \gamma) \le \gamma', \quad\quad\quad \ \  \text{ and } \quad\quad  (1 \pm \gamma)^{{\binom{k}{t}} + 1} = (1 \pm \gamma'),
\end{equation}
\begin{equation}\label{eq:gamma_3}
    (1 \pm 2\gamma) = (1 \pm \delta), \quad\quad \text{ and }  \quad\quad \gamma{\binom{k}{t}}2(1 - 2\gamma)^{-1} \le \delta.
\end{equation}
We do not attempt to optimize these calculations by removing redundant constraints. We will also fix $\ep_0(k,d,\delta) = \min\{\ep_1, \ep_2\}$ and suppose $\ep < \ep_0$.

Recall that $G(T)^*$ is a random subset of $G(T)$ where each $t$-clique is included with probability $p_T = \frac{\alpha n^t}{|G(T)|}$. Define
\[ G([k])^* = \left\{ \tilde{K} \in G([k]) :\  \forall\  T \in \binom{[k]}{t}, \ \forall \ \tilde{T} \in G(T),\  \tilde{T} \subseteq \tilde{K} \implies \tilde{T} \in G(T)^* \right\} .\]

Let $\mathcal{H}$ be the ${\binom{k}{t}}$-uniform hypergraph such that $V(\mathcal{H}) = \bigcup \{ G(T)^*  : T \in {\binom{k}{t}}\}$  and  every $\tilde{K} \in G([k])^*$ forms a hyperedge of $\mathcal{H}$, whose vertices are the $t$-cliques contained in $\tilde{K}$. Note that $\deg_{\cH}(\tilde{T})$ is the number of $\tilde{K} \in G([k])^*$ containing $\tilde{T}$.  

Now we  prove that $\deg_{\cH}(\tilde{T})$ is concentrated around its mean. Observe that $|G(T)^*|$ for $T \in \binom{K}{t}$ is the sum of independent indicator random variables. Applying the  Chernoff bound (Proposition~\ref{prop:chernoff}, with $\gamma$ serving the role of $\eta$), for every $T \in {\binom{k}{t}}$,  with probability at least $1 - \exp\left( - \Omega(n) \right)$, we have 
\begin{equation}\label{eq:conc_G(T)}
    |G(T)^*| = (1 \pm \gamma) p_T |G(T)| =  (1 \pm \gamma) \alpha n^t,
\end{equation}
where the last equality follows from the definition of $p_T$.

We now bound $|G(T)^*_{\bar{\gamma}}|$, where $G(T)^*_{\bar{\gamma}} = G(T)^* \cap G(T)_{\bar{\gamma}}$. If $|G(T)_{\bar{\gamma}}| \le n$, then as $t \ge 2$, and $\alpha \ge \alpha_0$, for sufficiently large $n$, we have 
\[ |G(T)^*_{\bar{\gamma}}| \le |G(T)_{\bar{\gamma}}| \le n \le \gamma \alpha_0n^t \le \gamma \alpha n^t.\]
If $|G(T)_{\bar{\gamma}}| \ge n$, then by a Chernoff bound, we have that, with high probability,
\[ |G(T)^*_{\bar{\gamma}}| \le (1 + \gamma) p_T |G(T)_{\bar{\gamma}}|.\]

Recall that by \eqref{eq:size-bad-G(S)} we have $|G(T)_{\bar{\gamma}}| \le \gamma^2 n^td_T$ and by \eqref{eq:sizeG(T)} we have $|G(T)| \ge (1 - \gamma)n^td_T$. Again, using the definition of $p_T$,  
\[ |G(T)^*_{\bar{\gamma}}| \le (1 + \gamma) \alpha n^t \frac{|G(T)_{\bar{\gamma}}|}{|G(T)|} \le \frac{(1 + \gamma) \gamma^2}{1 - \gamma}  \alpha n^t \le \gamma \alpha n^t, \]
where the last inequality follows from \eqref{eq:gamma_1}.  In both cases, if $n$ is sufficiently large, with high probability, we have
\begin{equation}\label{eq:conc_G(T)atypical}
    |G(T)^*_{\bar{\gamma}}| \le  \gamma \alpha n^t.
\end{equation}

  Now we estimate the degree of $\tilde{T} \in V(\cH)$. Let $\kappa_t$ be the number of $t$-cliques in $G$, i.e., $\kappa_t = \left|\tcliques{t}{G}\right|$. 
 Fix a $\tilde{T} \in G(T)$, let $\tilde{T}_1, \dots, \tilde{T}_{\kappa_t - 1}$ be an arbitrary ordering of the other   $t$-cliques of $G$. We define the function $$f_{\tilde{T}} : \{0,1\}^{\kappa_t - 1} \to \R,$$
where $f_{\tilde{T}}(x_{\tilde{T}_1}, \dots, x_{\tilde{T}_{\kappa_t - 1}})$ is equal to the number of $k$-cliques $\tilde{K} \in G(K)$ such that $\tilde{T} \subseteq \tilde{K}$ and $x_{\tilde{T}_i} = 1$ for $\tilde{T}_i \subseteq \tilde{K}$. 
For all $i$, let 
\begin{equation*}
  b_i = n^{k + z - 2t},  
\end{equation*}
where $|\tilde{T} \cap \tilde{T}_i| = z$, and note that $1 \leq z \leq t-1$.
By definition of $f_{\tilde{T}}$, for all $i \in [\kappa_t - 1]$, we have 
\begin{equation*}
|f_{\tilde{T}}(x_{\tilde{T}_1}, \dots, x_{\tilde{T}_{i-1}}, 1, x_{\tilde{T}_{i+1}}, \dots, x_{\tilde{T}_{\kappa_t - 1}}) - f_{\tilde{T}}(x_{\tilde{T}_1}, \dots, x_{\tilde{T}_{i-1}}, 0, x_{\tilde{T}_{i+1}}, \dots, x_{\tilde{T}_{\kappa_t - 1}})| \leq b_i.
\end{equation*}

 Also, for each nonnegative $z \leq t - 1$, the number of $\tilde{S} \in \tcliques{t}{G}$ such that $|\tilde{T} \cap \tilde{S}| = z$ is at most $\binom{t}{z}{\binom{k-t}{t - z}}n^{t-z} \le \binom{k}{t}n^{t-z}$. Therefore, we have 
\begin{equation}\label{eq:sum-ci^2}
  \sum_{i = 1}^{\kappa_t - 1} b_i^2 \leq \sum_{z = 0}^{t-1} \binom{k}{t} n^{t-z} \cdot n^{2(k+z - 2t)}  = \binom{k}{t}  \sum_{z = 0}^{t-1}n^{2k + z - 3t} \le t2^kn^{2k -2t - 1}.  
\end{equation} 
For each $i \in [\kappa_t - 1]$, let $X_i$ be the indicator random variable for the event  that $\tilde{T}_i$ is in $V(\cH)$. Hence, conditioning on $\tilde{T} \in V(\cH)$, 
\[ \deg_{\cH}(\tilde{T})  = f_{\tilde{T}}(X_1, \dots, X_{\kappa_t - 1}).\]
 For each $\tilde{K} \in \tcliques{k}{G}$, let $Y_{\tilde{K}}$ be the indicator random variable that   $\tilde{K}$'s  corresponding edge is  included in $\cH$, i.e.,  $X_i = 1$ for all $\tilde{T}_i \subseteq \tilde{K}$. As \[\pr\left(Y_{\tilde{K}} = 1 | \ \tilde{T} \in V(\cH) \right) = \prod_{\substack{S \in \binom{K}{t}\\ S \neq T}} p_S = \prod_{\substack{S \in \binom{K}{t}\\ S \neq T}} \frac{\alpha n^t}{|G(S)|},\] we have 
\begin{align}\label{eq:expectation-deg(T)}
& \E[\deg_{\cH}(\tilde{T})| \ \tilde{T} \in V(\cH)]\  =\   \sum_{\substack{\tilde{K} \in G(K)\\ \tilde{T} \subseteq \tilde{K}}} \E[Y_{\tilde{K}}|\ \tilde{T} \in V(\cH)]  
\ =\  \deg_{G}(\tilde{T})  \prod_{\substack{S \in \binom{K}{t} \\ S \neq T}} \frac{\alpha n^t}{|G(S)|} \notag\\
 =\  & \frac{1}{(1\pm \gamma^2)^{{\binom{k}{t}}-1}} \deg_{G}(\tilde{T}) \alpha^{\binom{k}{t}-1} \dfrac{d_T}{\hat{d}_{K}}\  = \ (1 \pm \gamma)^{{\binom{k}{t}} - 1} \deg_G(\tilde{T}) \alpha^{{\binom{k}{t}} - 1} \frac{d_T}{\hat{d}_K},
\end{align}
where the second to last equality follows from \eqref{eq:sizeG(T)}, i.e., $|G(S)| = (1 \pm \gamma^2)n^{|S|}d_S$ and the last equality follows from $\gamma$ being sufficiently small, see \eqref{eq:gamma_1}.  

We are now ready to establish that condition (i) of Theorem~\ref{thm:nibble} holds with high probability, that is, that almost all $\tilde{T} \in V(\cH)$ have nearly the same degree, which is
\[ D := n^{k-t} \alpha^{\binom{k}{t}-1} \frac{d_{K}}{\hat{d}_K} .\]
By the definition of $\gamma$-good, for every  $\tilde{T} \in G(T)^*_\gamma$, we have $\deg_G(\tilde{T}) = (1 \pm \gamma)n^{k-t} \frac{d_{K}}{d_T}$. It follows from \eqref{eq:expectation-deg(T)} that 
 
 \begin{align*}
    \E[\deg_{\cH}(\tilde{T})|\ \tilde{T} \in V(\cH)] = (1 \pm \gamma)^{{\binom{k}{t}} - 1} \deg_{G}(\tilde{T}) \alpha^{\binom{k}{t}-1} \frac{d_{T}}{\hat{d}_{K}} = (1 \pm \gamma)^{{\binom{k}{t}}}  n^{k-t} \alpha^{\binom{k}{t}-1} \frac{d_{K}}{\hat{d}_{K}}.
 \end{align*}
As $d_K/\hat{d}_K \ge 1$, we  conclude,
    \begin{align*}
        \E[\deg_{\cH}(\tilde{T})|\tilde{T} \in V(\cH)] \ge (1 - \gamma)^{\binom{k}{t}}\alpha^{\binom{k}{t}-1} n^{k-t}.
    \end{align*}
Combining with \eqref{eq:sum-ci^2}, we obtain 
\begin{equation*}
    \frac{\left(\E[\deg_{\cH}(\tilde{T})\big|\tilde{T} \in V(\cH)]\right)^2}{\sum_{i} b_i^2}\  \geq \  \frac{(1-\gamma)^{2\binom{k}{t}}\alpha^{2\binom{k}{t}-2}n^{2(k-t)}}{t2^kn^{2k-2t-1}}\  =\  (1-\gamma)^{2\binom{k}{t}}\alpha^{2\binom{k}{t}-2}\frac{n}{ t2^{k}} .
\end{equation*}

Conditioning on $\tilde{T} \in V(\cH)$, $\deg_{\cH}(\tilde{T})$ is a function of independent random variables as described before. By McDiarmid's inequality (Theorem~\ref{thm:McDiarmid}), we have 

\begin{align*}
\pr & \left(\left|\deg_{\cH}(\tilde{T}) - \E[\deg_{\cH}(\tilde{T})]\right|\  \geq \  \gamma \cdot  \E[\deg_{\cH}(\tilde{T})] \bigg| \tilde{T} \in V(\cH) \right)\\
\leq  & \  2 \exp\left( - {2\gamma^2 \left(\E[\deg_{\cH}(\tilde{T})|\tilde{T} \in V(\cH)]\right)^2}/\ {\sum_{i} b_i^2}\right)\\
\leq & \ 2 \exp\left(-2\gamma^2(1-\gamma)^{2\binom{k}{t}}\frac{\alpha^{2\binom{k}{t}-2}}{ t2^{k}} n\right).
\end{align*}
This shows that, conditioning on $\tilde{T} \in G(T)_{\gamma}^*$, with probability at least $1 - \exp(-\Omega(n))$, 

\begin{align*}
  \deg_{\cH}(\tilde{T}) &= (1 \pm \gamma) \E [\deg_{\cH}(\tilde{T})|\tilde{T} \in V(\cH)]  =   (1 \pm \gamma)^{{\binom{k}{t}} + 1} n^{k-t} \alpha^{\binom{k}{t}-1} \frac{d_{K}}{\hat{d}_{K}} = (1 \pm \gamma)^{{\binom{k}{t}} + 1} D.
\end{align*}

Hence,  with probability $1 - {\binom{k}{t}}n^t \exp\left(-\Omega(n)\right)$, we have $\deg_{\cH}(\tilde{T}) = (1 \pm \gamma)^{{\binom{k}{t}} + 1}D$ for all $\gamma$-good cliques $\tilde{T} \in V(\cH)$, which are all but $\gamma \binom{k}{t} \alpha n^t$ (by~\eqref{eq:conc_G(T)atypical}) of the $(1\pm \gamma) \binom{k}{t} \alpha n^t$ total vertices (by~\eqref{eq:conc_G(T)}), verifying condition (i) of Theorem~\ref{thm:nibble}.

Next, we prove that condition (ii) of Theorem \ref{thm:nibble} holds for \textit{every} $\tilde{T} \in G(T)^*$. First, if there are at most $D =  n^{k-t} \alpha^{\binom{k}{t}-1} \frac{d_{K}}{\hat{d}_{K}}$ $k$-cliques in $G$ containing  $\tilde{T}$, then  we are done. Thus, we only need to consider the case when $\deg_G(\tilde{T}) \geq D.$ It follows from \eqref{eq:expectation-deg(T)} that 
\[   \E[\deg_{\cH}(\tilde{T})| \tilde{T} \in V(\cH)] \geq (1 - \gamma)^{{\binom{k}{t}} - 1} \deg_G(\tilde{T}) \alpha^{{\binom{k}{t}} - 1} \frac{d_T}{\hat{d}_K} \ge (1 - \gamma)^{{\binom{k}{t}} - 1} n^{k-t} \alpha^{2{\binom{k}{t}} - 2} \frac{d_Td_K}{\hat{d}_K^2}.\]
Hence, in this case, by \eqref{eq:sum-ci^2}, 
\begin{align*}
   &  {\left( \E[\deg_{\cH}(\tilde{T})| \tilde{T} \in V(\cH)]\right)^2}/\ {\sum_{i} b_i^2} \geq \ \left( (1 - \gamma)^{{\binom{k}{t}} - 1} n^{k-t} \alpha^{2{\binom{k}{t}} - 2 } {d_Td_K}{\hat{d}_K^{-2}}\right)^2t^{-1}2^{-k}n^{2t -2k +1}\\  = & \ (1 - \gamma)^{2{\binom{k}{t}} - 2} \alpha^{4\binom{k}{t}-4}2^{-k}t^{-1}{d_T^2d_K^2}\hat{d}_K^{-4} n
    \ge \ \alpha^{4{\binom{k}{t}} - 4} 2^{-2{\binom{k}{t}} -k + 2} t^{-1}{d_T^2d_K^2}\hat{d}_K^{-4} n,
\end{align*}
where the last inequality follows as $\gamma \le 1/2$.

Therefore, by McDiarmid's inequality and \eqref{eq:expectation-deg(T)}, with probability at least $1 - \exp(-\Omega(n))$,
\begin{align*}
    \deg_{\cH}(\tilde{T}) \leq (1+\gamma) \E[ \deg_{\cH}(\tilde{T})] &\le (1 + \gamma)^{{\binom{k}{t}}}\deg_{G}(\tilde{T}) \alpha^{\binom{k}{t}-1} \frac{d_{T}}{\hat{d}_{K}} \\
    &\leq (1 + \gamma)^{{\binom{k}{t}}} n^{k-t} \alpha^{\binom{k}{t}-1} \frac{d_{T}}{\hat{d}_{K}}  \quad \quad\text{[because  $\deg_{G}(\tilde{T}) \leq n^{k-t}$]}  \\ 
    &=  (1 + \gamma)^{{\binom{k}{t}}}  \frac{d_T}{d_K\hat{d}_{K}} D \leq 2^{\binom{k}{t}} d^{-2\binom{k}{2}\binom{k-2}{t-2}} D = CD, 
\end{align*}
where the last inequality follows from $\gamma \le 1$ and $1 \ge d_{i,j} \geq d$ for all $1 \leq i < j \leq k$, and the last equality is the definition of $C$ from~\eqref{eq:C_and_delta_selection}. Taking a union bound over all $O(n^t)$ choices of $\tilde{T}$ shows that the maximum degree of $\cH$ is bounded by $CD$, that is, condition (ii) of Theorem~\ref{thm:nibble} holds, with high probability.

Finally we establish condition  (iii) of Theorem~\ref{thm:nibble}, that is, for every $T_1, T_2 \in \binom{[k]}{t}$,  the co-degree of $\tilde{T}_1 \in G(T_1)^*$ and $\tilde{T}_2 \in G(T_2)^*$ is small compared to $D$. Their co-degree is the number of $\tilde{K} \in G(K)^*$ which contain $\tilde{T}_1 \cup \tilde{T}_2$, which is trivially at most $n^{k-|T_1 \cup T_2|} \leq n^{k-t-1}$. Thus, for  $n$ sufficiently large  in terms of $d$, $k$, and $\alpha_0$, and $\delta'$, we have 
\[ n^{k-t-1} < \gamma' n^{k-t} \alpha_0^{\binom{k}{t}-1} d^{\binom{k}{2}\left(1-\binom{k-2}{t-2}\right)}  \leq \gamma' n^{k-t} \alpha_0^{\binom{k}{t}- 1} \frac{d_{K}}{\hat{d}_{K}}   \le \gamma' D,\]
establishing condition (iii) of Theorem~\ref{thm:nibble}.

Thus for $n$ sufficiently large with respect to $d$, $k$, and $\alpha_0$, by Theorem \ref{thm:nibble}, we  conclude that $\mathcal{H}$ has an edge packing of size at least $(1 - \delta') |V(\mathcal{H})|/{\binom{k}{t}}$ with high probability. Let $\mathcal{S}_1$ be the set of $t$-cliques covered in this packing, and let $\mathcal{S}_2 = \bigcup_{T \in \binom{[k]}{t}} \left(G(T)^* \setminus \mathcal{S}_1 \right)$.

We now show that $\mathcal{S}_1$ and $\mathcal{S}_2$ satisfy properties (i), (ii), and (iii) of Lemma~\ref{lem:random_nibble}, completing the proof. By definition of $\mathcal{S}_1$, (i) holds. 
As $\gamma$ is sufficiently small (see \eqref{eq:gamma_3}), we have 
\[ |\mathcal{S}_1| = \binom{k}{t} (1\pm \delta') |V(\mathcal{H})|/\binom{k}{t} = (1 \pm \delta){\binom{k}{t}}\alpha n^t,\]
implying (ii). Since $\mathcal{S}_2$ is the set of uncovered vertices of $V(\cH)$, using that $\gamma$ is sufficiently small (see \eqref{eq:gamma_3}), we may conclude
\[ |\mathcal{S}_2| \leq (1+\gamma) \binom{k}{t} \alpha n^t - |\mathcal{S}_1| \le \delta \alpha n^t. \qedhere\] 
\end{proof}

\subsection{Fractional to integer covers for cleaned graphs}\label{sec:asymptotic:frac_to_int}

Recall that an $(n,\ell,\ep,d)$-graph $G$ has $n$ vertices which can be equi-partitioned into independent sets $V_1, \dots, V_{\ell}$ such that for every $1 \leq i < j \leq \ell$, either $(V_i,V_j)$ is an $\ep$-regular pair with density at least $d$ or $V_i \cup V_j$ is an independent set. Here we show that such graphs have fractional $t$-clique cover number and decomposition number close to their integer $t$-clique cover number and decomposition number, respectively. Furthermore, this closeness is effectively bounded in terms of the clique number of the graph. 

\begin{lemma}\label{lem:frac_to_int}
Let $t \ge {2}$ be an integer and let $\mathbf{c} = (c_i)_{i = 1}^{\infty}$ be a sequence of positive real numbers, where each $c_i\ge 1$ for all $i \ge t$. For every integer $b$ and pair of  real numbers $0<d,\rho < 1$, there exists $\ep_0 = \ep_0(b,d,(c_i)_{i = t}^b,\rho) > 0$, such that for every $\ell$, there exists $n_0 = n_0(b,d,(c_i)_{i = t}^b,\rho,\ell)$ 
such that the following holds. Let $G$ be an $(n,\ell,\ep,d)$-graph such that $n \ge n_0$, $\omega(G) \leq b$, and $\ep < \ep_0$. Then
\[ \coveropt{t}{G,\mathbf{c}} \leq \coveroptfrac{t}{G,\mathbf{c}} + \rho n^t \]
and
\[ \decompositionopt{t}{G,\mathbf{c}} \leq \decompositionoptfrac{t}{G,\mathbf{c}} + \rho n^t.\]
\end{lemma}

\begin{proof}
We prove the lemma for covers and decompositions at the same time, noting the relevant differences when they arise.

Let $V_1, \dots, V_{\ell}$ be a vertex partition of $G$ witnessing that $G$ is an $(n,\ell,\ep,d)$-graph.  Let $R$ be the reduced graph of $G$, that is, the graph on $[\ell]$ where $i$ is adjacent to $j$ if and only if $(V_i,V_j)$ is a regular pair of  density at least $d$, and denote by $d_{i,j}$  the density of $(V_i,V_j)$. Let $\rho \in (0,1)$ be given.

Let $f : \cliques{G} \to \mathbb{R}_{\geq 0}$ be an optimal fractional $t$-clique cover/decomposition of $G$. Note that as $\mathbf{c}$ is a positive sequence, $f$ is supported on cliques with at least $t$ vertices. Let $\hat{f} : \cliques{R} \to \mathbb{R}_{\geq 0}$ such that
\[ \hat{f}(K) = \sum_{K' \in G(K)} f(K') .\]
We will use $\hat{f}$ as a guide for producing a nearly optimal integer $t$-clique cover/decomposition of $G$ via Lemma~\ref{lem:random_nibble}. Essentially, if $\hat{f}(K)$ is not too small, then we will apply Lemma~\ref{lem:random_nibble} to the graph induced on $G$ by the parts corresponding to $K$ with probabilities $p_T = \frac{\hat{f}(K)}{C|G(T)|}$ for some suitably large $C$ not depending on $T$ or $K$. For each $T \in \tcliques{t}{G}$, we couple our choices of random subsets of $G(T)$ for different $K$ in such a way that they cover/partition each $G(T)$. Care was taken in the set-up of Lemma~\ref{lem:random_nibble} so that it applies as long as $\hat{f}(K)/C$ is bounded away from zero independently of $n$, but $\hat{f}(K)$ can depend on any of the other parameters, in particular, the regularity parameters $\ell$ and $\ep$. When $\hat{f}(K)$ is too small, we will forego applying Lemma~\ref{lem:random_nibble} and instead just put each $t$-clique of $G$ randomly assigned to $K$ into the integer cover/decomposition.

We first show that the $\frac{\hat{f}(K)}{C|G(T)|}$ will be suitable choices for the probabilities $p_T$ in our applications of Lemma~\ref{lem:random_nibble}. We claim that
\begin{equation}\label{eq:fhat_upper}
    \hat{f}(K) \leq C \cdot |G(T)|
\end{equation}
for every $K \in \cliques{R}$ and every $T \in \tcliques{t}{R}$ with $T \subseteq K$, where
\[ C = C(d,b,t) := 2 \binom{b}{t} d^{-\binom{t}{2}} .\]
First, observe that by Lemma~\ref{lem:regularity_counting} for sufficiently small $\ep$ with respect to $d$ and $b$, we have
\begin{equation}\label{eq:G(T)_bounds}
    \frac{1}{2} d^{\binom{t}{2}} (n/\ell)^t  \le |G(T)| \le (n/\ell)^t .
\end{equation}
We say that $T' \in \tcliques{t}{G}$ is \emph{tight} with respect to $f$ if
\[ \sum_{\substack{K' \in \cliques{G} \\ T' \subseteq K'}} f(K') = 1 .\]
Note that if $f$ is a decomposition, then all $T' \in \tcliques{t}{G}$ are tight with respect to $f$, by definition.
If $f$ is an optimal fractional $t$-clique cover, then every $K'$ in the support of $f$ must contain a tight $t$-clique; otherwise, we could reduce $f(K')$ slightly, keeping $f$ the same everywhere else, and obtain another fractional $t$-clique cover of $G$ with slightly lower total weight. Thus, for both the cover and decomposition cases, for every  $K \in \cliques{R}$,
\[ \sum_{K' \in G(K)} f(K') \leq \sum_{\substack{T^* \in \K_t(R) \\ T^* \subseteq K}} \sum_{\substack{T' \in G(T^*) \\ T' \text{tight}}} \sum_{\substack{K' \in \K(G) \\ T' \subseteq K'}} f(K') \leq \sum_{\substack{T^* \in \K_t(R) \\ T^* \subseteq K}} |G(T^*)| \le \binom{b}{t} (n/\ell)^t .\]
By~\eqref{eq:G(T)_bounds}, we have that
\[ \hat{f}(K) = \sum_{K' \in G(K)} f(K') \le \binom{b}{t} \frac{n^t}{\ell^t} \le 2\binom{b}{t}d^{-\binom{t}{2}}|G(T)| =  C\cdot |G(T)| ,\]
that is, \eqref{eq:fhat_upper} holds.

Additionally, note that, for every $T \in \tcliques{t}{R}$,
\begin{equation}\label{eq:fhatsum}
\sum_{\substack{K \in \cliques{R}\\ T \subseteq K}} \hat{f}(K) = \sum_{\substack{K \in \cliques{R}\\ T \subseteq K}} \sum_{K' \in G(K)} f(K') = \sum_{T' \in G(T)} \sum_{K' \supseteq T'} f(K') \ge \sum_{T' \in G(T)} 1  = |G(T)|,
\end{equation}
where the first equality follows from the definition of $\hat{f}$, the second equality from the fact that $G$ is an $(n,\ell,\ep,d)$-graph, and the inequality from $f$ being a $t$-clique cover. When $f$ is additionally a decomposition, this inequality is in fact an equality (furthermore, in this case, \eqref{eq:fhatsum} implies \eqref{eq:fhat_upper} with $C=1$, although we will not use this fact).

Now we use $\hat{f}$ to choose random subsets of $G(T)$ for $T \in \tcliques{t}{R}$ on which to apply Lemma~\ref{lem:random_nibble}.
For each $T \in \tcliques{t}{R}$, we arbitrarily cover/partition the interval $[0,1]$ with intervals $I^i_{T,K}$, where $i \in [C]$ and $K \in \cliques{R}$ with $T \subseteq K$ such that $|I_{T,K}^i|=\frac{\hat{f}(K)}{C|G(T)|}$. Such a covering/partition exists because, by~\eqref{eq:fhat_upper} we have $\frac{\hat{f}(K)}{C|G(T)|} \leq 1$, and by~\eqref{eq:fhatsum}, we have
\[ \sum_{K \supseteq T} \sum_{i \in [C]}
\frac{\hat{f}(K)}{C|G(T)|} = \ \sum_{K \supseteq T} \frac{\hat{f}(K)}{|G(T)|} \geq 1 ,\]
where the inequality is in fact equality in the case of decompositions.
For each $T' \in \tcliques{t}{G}$, choose $x_{T'} \in [0,1]$ independently and uniformly at random. For $T \in \tcliques{t}{R}$, $K \in \cliques{R}$ with $T \subseteq K$, and $i \in [C]$, let
\[ S_{T,K}^i = \left\{ T' \in G(T) : x_{T'} \in I_{T,K}^i \right\} .\]
Observe that since the $I_{T,K}^i$ cover/partition $[0,1]$ for every $T \in \tcliques{t}{R}$, we have that $\{S_{T,K}^i\}_{i,K}$ covers/partitions $G(T)$ for every $T \in \tcliques{t}{R}$. Furthermore, for each fixed $K \in \cliques{R}$ and $i \in [C]$, $S_{T,K}^i$ is a random subset of $G(T)$ where we include each element independently with probability $|I^i_{T,K}|  = \frac{\hat{f}(K)}{C|G(T)|}$, independently of the other $S_{T',K}^i$ with $T' \subseteq K$.

Let
\[ \alpha_0 = \alpha_0(\rho,b,\ell,(c_i)_{i = t}^b,d) := \min\left\{ \frac{\rho}{3} \left( \binom{\ell}{\leq b} C c_t 2 \binom{b}{t} \right)^{-1} , d^{\binom{t}{2}} \right\} ,\]
and $\delta$ be chosen smaller  than $\min_{t\leq j\leq b} \{\frac{\rho c_j}{3c_t^2}, \frac{\rho}{3c_t}\}$. Since $\alpha_0$ depends only on $\rho$, $(c_i)_{i = t}^b$, $d$, $\ell$, $b$ and $t$, Lemma~\ref{lem:random_nibble} holds for $n$ sufficiently large: for each $K \in \cliques{R}$ and $i \in [C]$, if $\frac{\hat{f}(K)}{C} \geq \alpha_0 n^t$, then with probability at least $1-1/n$ there exists a partition $\mathcal{S}^1_{K,i}, \mathcal{S}^2_{K,i}$ of $\bigcup_{T \subseteq K} S_{T,K}^i$, such that
\begin{itemize}
\item $\mathcal{S}^1_{K,i}$ can be partitioned into sets of the form $\{T' : T' \subseteq K'\}$ for $K' \in G(K)$,
\item $|\mathcal{S}^1_{K,i}| \le (1 + \frac{\rho}{3c_t})\binom{|K|}{t} \frac{ \hat{f}(K)}{C}$, and
\item $|\mathcal{S}^2_{K,i}| \le  \frac{\rho c_{|K|}}{3c_{t}^2} \frac{\hat{f}(K)}{C}$.
\end{itemize}
Since there are at most $\binom{\ell}{\leq b} C$ choices of $K \in \cliques{R}$ and $i \in [C]$ and $n$ is sufficiently large, we may take a union bound over all such choices so that such a partition $\mathcal{S}^1_{K,i}, \mathcal{S}^2_{K,i}$ exists for each $K \in \cliques{R}$ and $i \in [C]$ with high probability. When $\frac{\hat{f}(K)}{C} < \alpha_0 n^t$, we simply set $\mathcal{S}^1_{K,i} = \emptyset$ and $\mathcal{S}^2_{K,i} = \bigcup_{T \subseteq K} S_{T,K}^i$. For $K \in \cliques{R}$, the expected size of $\mathcal{S}^2_{K,i}$ is $\binom{|K|}{t} \frac{\hat{f}(K)}{C}$,  so if $\frac{\hat{f}(K)}{C} < \alpha_0 n^t$, then by the Chernoff bound (Proposition~\ref{prop:chernoff}),
\[ |\mathcal{S}^2_{K,i}| \leq 2 \binom{|K|}{t} \alpha_0 n^t \]
with probability at least $1-1/n$. This means that we are able to take a union bound as before. 

By construction, every $T' \in \tcliques{t}{G}$ appears in at least one $\mathcal{S}_{K,i}^j$, where $K \in \cliques{R}$, $i \in [C]$, and $j \in \{1,2\}$, and in exactly one in the case of decompositions. Thus the union of the given clique decompositions of all the $\mathcal{S}^1_{K,i}$ together with the $t$-cliques of all the $\mathcal{S}^2_{K,i}$ forms an integer $t$-clique cover/decomposition of $G$. The $\mathbf{c}$-cost of this cover/decomposition is at most
\begingroup
\allowdisplaybreaks
\begin{align*}
&\sum_{\substack{K \in \cliques{R} \\ \hat{f}(K)/C \geq \alpha_0 n^t}} \sum_{i \in [C]} \left( c_{|K|} \frac{|\mathcal{S}_{K,i}^1|}{\binom{|K|}{t}} + c_t |\mathcal{S}_{K,i}^2| \right) + \sum_{\substack{K \in \cliques{R} \\ \hat{f}(K)/C < \alpha_0 n^t}} \sum_{i \in [C]} c_t |\mathcal{S}_{K,i}^2| \\
&\leq \sum_{\substack{K \in \cliques{R} \\ \hat{f}(K)/C \geq \alpha_0 n^t}} \sum_{i \in [C]} \left( c_{|K|} \left(1+ \frac{\rho}{3c_t} \right) + c_t \frac{\rho c_{|K|}}{3c_t^2} \right) \frac{\hat{f}(K)}{C} + \sum_{\substack{K \in \cliques{R} \\ \hat{f}(K)/C < \alpha_0 n^t}} \sum_{i \in [C]} c_t \cdot 2 \binom{|K|}{t} \alpha_0 n^t \\
&= \sum_{\substack{K \in \cliques{R} \\ \hat{f}(K)/C \geq \alpha_0 n^t}} c_{|K|} \left(1+\frac{2\rho}{3c_t}\right) \hat{f}(K) + \sum_{\substack{K \in \cliques{R} \\ \hat{f}(K)/C < \alpha_0 n^t}} C c_t \cdot 2 \binom{|K|}{t} \alpha_0 n^t \\
&\leq \ (1+2\rho/3c_t) \| \hat{f} \|_\mathbf{c} + \binom{\ell}{\leq b} C c_t 2 \binom{b}{t} \alpha_0 n^t \ \leq \ \norm{f}_\mathbf{c} + \rho n^t,
\end{align*}
\endgroup
where the final inequality follows from $\| \hat{f} \|_\mathbf{c} = \norm{f}_\mathbf{c} \leq c_t n^t$ and the definition of $\alpha_0$.
\end{proof}

\paragraph{Remark.} There are two main technical challenges in the proof of Lemma~\ref{lem:frac_to_int} that may be difficult to see through the details. First, we need the cliques in the cover/decomposition to have size bounded independently of $\ep$ (and hence $\ell$) so that Lemma~\ref{lem:random_nibble} may be applied. Second, when we are not able to apply Lemma~\ref{lem:random_nibble}, that is, when $\hat{f}(K)/C < \alpha_0 n^t$, we need that the total cost of the $t$-cliques covered is small. Since there may be many such $K$ (polynomial in $\ell$, which depends on $\ep$), we need to be able to take $\alpha_0$ close to $0$ depending on $\ep$ and $\ell$. This requires a delicate order of quantifiers in Lemma~\ref{lem:random_nibble}.

\subsection{Proof of Theorem~\ref{thm:asymptotic_bound}}\label{sec:asymptotic:final_proof}

We first establish some preliminary lemmas which allow us to reduce to `cleaned' graphs with bounded clique number.

\begin{lemma}\label{lem:remove_edges_cover}
Let $G$ be an $n$-vertex graph, $t \geq 2$, and $\mathbf{c}$ be a nonnegative real sequence. If $G'$ is a subgraph of $G$, then
\[ \coveropt{t}{G,\mathbf{c}} \leq \coveropt{t}{G',\mathbf{c}} + c_t |E(G) \setminus E(G')| n^{t-2} \]
and
\[ \decompositionopt{t}{G,\mathbf{c}} \leq \decompositionopt{t}{G',\mathbf{c}} + c_t |E(G) \setminus E(G')| n^{t-2}. \]
\end{lemma}

\begin{proof}
Let $f' : \cliques{G'} \to \mathbb{R}_{\geq 0}$ be an integer $t$-clique cover/decomposition of $G'$. Define $f : \cliques{G} \to \mathbb{R}_{\geq 0}$ by
\[ f(K) = \begin{cases}
            f'(K) &\text{ if } K \in \cliques{G'} \subseteq \cliques{G} ,\\
            1 &\text{ if } K \in \tcliques{t}{G} \setminus \tcliques{t}{G'},\\
            0 &\text{ otherwise}.
        \end{cases} \]
It is easy to check that $f$ is an integer $t$-clique cover/decomposition of $G$. Since 
\[ \norm{f}_{\mathbf{c}} \leq \norm{f'}_{\mathbf{c}} + c_t |\tcliques{t}{G} \setminus \tcliques{t}{G'}| \leq \norm{f'}_{\mathbf{c}} + c_t|E(G) \setminus E(G')| n^{t-2} ,\]
we reach the desired conclusion.
\end{proof}

\begin{lemma}\label{lem:no_big_cliques}
Let $G$ be an $n$-vertex graph, and $b \geq 2$ such that $n \ge 4b$. Then there exists an $n$-vertex subgraph $G'$ of $G$ such that $\omega(G') \leq b$ and $|E(G) \setminus E(G')| \leq n^2/b$.
\end{lemma}

\begin{proof}
Let $H$ be isomorphic to $\turangraph{n}{b}$ on $V(G)$, and let $G' = G \cap H$. Then $\omega(G') \leq \omega(H) = b$ and
\[ |E(G) \setminus E(G')| \leq |K_n \setminus H| \leq b \frac{\ceil{n/b}^2}{2} \leq \frac{n^2}{b} . \qedhere\]
\end{proof}

\begin{proof}[Proof of Theorem~\ref{thm:asymptotic_bound}]
We prove the theorem for covers, although the proof goes through identically for decompositions in place of covers.

Without loss of generality, we may suppose  $\rho < 1$. Let $b$ be a positive integer such that $6 c_t/\rho \le b$. Choose a real number $d>0$ and positive integer $m_0$ such that $1/m_0 + d \le \frac{\rho}{6c_t}$. Let $\ep_0(b,d,(c_i)_{i = t}^b,\rho/6)$ be as in Lemma~\ref{lem:frac_to_int} and let $\ep < \min\{\ep_0(b,d,(c_i)_{i = t}^b,\frac{\rho}{6}), \frac{\rho}{6c_t}\}$. Let $n_0(b,d,(c_i)_{i = t}^b,\rho/6,\ell)$ be as in Lemma~\ref{lem:frac_to_int}. Let $G$ be an $n$-vertex graph such that $n \ge \max\{4b,\allowbreak n_0(b,d,(c_i)_{i = t}^b,\rho/6,\ell)(1 - \ep)^{-1}\}$.  By Lemma~\ref{lem:no_big_cliques}, there exists an $n$-vertex subgraph $G'$ of $G$ such that $\omega(G') \le b$ and $|E(G) \setminus E(G')| \le n^2/b$. By Lemma~\ref{lem:cleaning}, $G'$ has a subgraph $\tilde{G}$ such that $\tilde{G}$ is an $(\tilde{n},\ell,\ep,d)$-graph, $\tilde{n} = |V(\tilde{G})| \ge (1 - \ep)n$, and $|E(G')\setminus E(\tilde{G})| \le (2\ep + 1/m_0 + d)n^2$  where $m_0 \le \ell \le M(\ep)$. As $\ep < \ep_0$ and $\tilde{n} = (1-\ep)n \ge n_0$, by Lemma~\ref{lem:frac_to_int} and Theorem~\ref{thm:frac},
\begin{align*}
    \coveropt{t}{\tilde{G},\mathbf{c}} \le \coveroptfrac{t}{\tilde{G},\mathbf{c}} + \rho \tilde{n}^t/6 
\end{align*}
Let $S = V(G') \setminus V(\tilde{G})$. As $|V(G') \setminus V(\tilde{G})| \le \ep n$ and by  Lemma~\ref{lem:cleaning}, $|E(G' - S) \setminus E(\tilde{G})| \le (2 \ep + 1/m_0 + d)n^2$. Then, Lemma~\ref{lem:remove_edges_cover} implies that
\begin{align*}
    \coveropt{t}{G',\mathbf{c}} \le \coveropt{t}{\tilde{G},\mathbf{c}} + c_t(3 \ep + 1/m_0 + d)n^t \le \coveropt{t}{\tilde{G},\mathbf{c}} + 4 \rho n^t/6.
\end{align*}
Finally, as $|E(G) \setminus E(G')| \le n^2/b$ by Lemma \ref{lem:remove_edges_cover}, 
\begin{align*}
    \coveropt{t}{G,\mathbf{c}} \le \coveropt{t}{G',\mathbf{c}} + c_tn^t/b \le \coveropt{t}{G',\mathbf{c}} + \rho n^t/6.
\end{align*}
Thus we conclude
\[ \coveropt{t}{G,\mathbf{c}} \le \coveropt{t}{\tilde{G},\mathbf{c}} + \rho n^t. \qedhere \]

\end{proof}

\section{Integrality gaps and packings}\label{sec:gap}

Let $G$ be a graph. In our proof of Theorem~\ref{thm:asymptotic_bound}, we prove that for a suitable cleaned graph $\tilde{G}$, $\tcoveroptfrac{\tilde{G},\mathbf{c}}$ is close to $\tcoveropt{G,\mathbf{c}}$ in an appropriate sense. The ratio/difference between $\coveroptfrac{t}{G,\mathbf{c}}$ and $\coveropt{t}{G,\mathbf{c}}$ ($\decompositionoptfrac{t}{G,\mathbf{c}}$ and $\decompositionopt{t}{G,\mathbf{c}}$) is known as the multiplicative/additive \emph{integrality gap} for the $t$-clique cover (decomposition) number. 

\subsection{Upper bounds}

Theorem~\ref{thm:asymptotic_bound_decomposition} states that the additive integrality gap for decompositions is $o(n^t)$. To prove this, we first need an inverse version of Lemma~\ref{lem:remove_edges_cover} for the decomposition case.

\begin{lemma}\label{lem:remove_edges_decomposition}
    Let $G$ be an $n$-vertex graph, let $t \ge 2$ and $\mathbf{c}$ be a nonnegative real sequence such that $c_{i+1} - c_i \leq c_t$ for all $i \geq t$. If $G'$ is a subgraph of $G$, then for all integers $b \ge t+1$,
    \[ \decompositionoptfrac{t}{G',\mathbf{c}} \le \decompositionoptfrac{t}{G,\mathbf{c}} + \frac{2c_tn^t}{(b-t)^{t-1}} + c_t\binom{b}{t}|E(G) \setminus E(G')|n^{t-2}. \]
\end{lemma}

\begin{proof}
Let $b \ge t+1$. Let $g : \cliques{G} \to \mathbb{R}_{\ge 0}$ be an optimal fractional $t$-clique decomposition of $G$.  Let $\tcliques{>b}{G} = \{K \in \cliques{G} : |K| > b\}$, and recall $\tcliques{b}{G}$ is the set of all cliques of $G$ of size $b$.  We modify $g$ in the following way. Consider the map $\hat{g} : \cliques{G} \to \mathbb{R}_{\ge 0}$, such that
    \begin{align*}
       \hat{g}(Q) = \begin{cases}
             0 &\text{ if } Q \in \tcliques{>b}{G},\\
            g(Q) &\text{ if } Q \in \tcliques{<b}{G},\\
            g(Q) + \sum\limits_{\substack{K \in \tcliques{>b}{G}\\ Q \subseteq K}} g(K)\binom{|K| - t}{b-t}^{-1}  &\text{ if } Q \in \tcliques{b}{G}.
    \end{cases} 
    \end{align*}
    Essentially, we take each clique on more than $b$ vertices and replace it by all of its $b$-cliques, appropriately weighted. Note by construction, $\hat{g}$ has support on cliques of size at most $b$, and $\hat{g}$ is a $t$-clique decomposition of $G$, since for all  $T \in \tcliques{t}{G}$, we have
    \begin{align*}
        \sum_{\substack{Q \in \cliques{G} \\ T \subseteq Q}} \hat{g}(Q) &= \sum_{\substack{Q \in \tcliques{b}{G}\\ T \subseteq Q}}\sum_{\substack{K \in \tcliques{>b}{G}\\ T \subseteq Q \subseteq K}} g(K)\binom{|K| - t}{b-t}^{-1} - \sum_{\substack{K \in \tcliques{>b}{G}\\ T \subseteq K}} g(K) + \sum_{\substack{Q \in \cliques{G} \\ T \subseteq Q}} g(Q)\\
        &= \sum_{\substack{K \in \tcliques{>b}{G}\\ T \subseteq K}} g(K) -   \sum_{\substack{K \in \tcliques{>b}{G}\\ T \subseteq K}} g(K) + \sum_{\substack{Q \in \cliques{G} \\ T \subseteq Q}} g(Q) = \sum_{\substack{Q \in \cliques{G} \\ T \subseteq Q}} g(Q) = 1.
    \end{align*}
Indeed, for every fixed $K\in\tcliques{>b}{G}$ containing $T$,
there are exactly
$\binom{|K|-t}{b-t}$ cliques
$Q\in\tcliques{b}{G}$ satisfying $T\subseteq Q\subseteq K$.
Thus, after reversing the order of the double sum above, this multiplicity
cancels the factor $\binom{|K|-t}{b-t}^{-1}$.
    Furthermore,
    \begingroup
\allowdisplaybreaks
    \begin{align}
    \norm{\hat{g}}_{\mathbf{c}} &= \norm{g}_{\mathbf{c}} + \sum_{Q \in \tcliques{b}{G}}\sum_{\substack{K \in \tcliques{>b}{G}\\Q \subseteq K}} c_b g(K)\binom{|K| - t}{b-t}^{-1} - \sum_{K \in \tcliques{>b}{G}} c_{|K|}g(K) \notag\\
    &\le \norm{g}_{\mathbf{c}} + c_b \sum_{Q \in \tcliques{b}{G}}\sum_{\substack{K \in \tcliques{>b}{G}\\  Q \subseteq K}}  g(K)\binom{|K| - t}{b-t}^{-1} \notag\\
    &= \norm{g}_{\mathbf{c}} + c_b \sum_{K \in \tcliques{>b}{G}}  g(K)\binom{|K|}{b}\binom{|K| - t}{b-t}^{-1} \notag\\
    &\le \norm{g}_{\mathbf{c}} + (b-t+1)c_t \sum_{K \in \tcliques{>b}{G}}  g(K)\binom{|K|}{b}\binom{|K| - t}{b-t}^{-1} \label{eq:decomposition_edge_deletion_equation}\\
    &= \norm{g}_\mathbf{c} + (b-t+1)c_t \sum_{K \in \tcliques{>b}{G}}  g(K)\binom{|K|}{t}\binom{b}{t}^{-1} \notag\\
    &\leq \norm{g}_\mathbf{c} + (b-t+1)c_t |\tcliques{t}{G}| \binom{b}{t}^{-1} \label{eq:decomposition_edge_deletion_equation2}\\
    &\le \norm{g}_{\mathbf{c}} + c_t(b-t+1)\frac{n^t}{(b-t)^t}
    \le \norm{g}_{\mathbf{c}} + 2c_t\frac{n^t}{(b-t)^{t-1}}, \notag
    \end{align}
    \endgroup
    where \eqref{eq:decomposition_edge_deletion_equation} follows from $c_{i+1} - c_i \leq c_t$ for all $i \geq t$, and \eqref{eq:decomposition_edge_deletion_equation2} follows from $g$ being a $t$-clique decomposition, meaning that each $t$-clique of $G$ is covered by cliques of weight totaling exactly one.

    Let $\mathcal{Q}$ be the set of cliques of $G$, which contain an edge of $E(G) \setminus E(G')$, i.e.,
    \[ \mathcal{Q} = \{ K \in \cliques{G} : E(K) \cap (E(G) \setminus E(G')) \neq \emptyset\}.\]
    We modify $\hat{g}$ to a $t$-clique decomposition $g'$ of $G'$ by replacing each clique of $\mathcal{Q}$ with its $t$-cliques: 
    \begin{align*}
       g'(Q) = \begin{cases}
             0 &\text{ if } Q \in \mathcal{Q},\\
            \hat{g}(Q) &\text{ if } Q \in \cliques{G} \setminus( \mathcal{Q} \cup \tcliques{t}{G}),\\
            \hat{g}(Q) + \sum\limits_{\substack{K \in \mathcal{Q}\\ Q \subseteq K}} \hat{g}(K)  &\text{ if } Q \in  \tcliques{t}{G} \setminus \mathcal{Q}.
    \end{cases} 
    \end{align*}
    Indeed, $g'$ has support on $\cliques{G'}$ and for any $T \in \tcliques{t}{G'}$, we have, 
    \[ \sum_{\substack{Q \in \cliques{G'} \\ T \subseteq Q}} g'(Q) = \sum_{\substack{Q \in \cliques{G} \setminus \mathcal{Q}\\ T \subsetneq Q}}\hat{g}(Q) + \sum_{\substack{Q \in \mathcal{Q}\\ T \subsetneq Q}} \hat{g}(Q) + \hat{g}(T)   = \sum_{\substack{Q \in \cliques{G}\\ T \subseteq Q}} \hat{g}(Q) = 1. \]
    Thus, restricting the domain of $g'$ to $\cliques{G'}$ yields a fractional $t$-clique decomposition of $G'$. Now, as $\hat{g}$ was supported on cliques of size at most $b$, we have, 
    \begin{align*}
        \norm{g'}_{\mathbf{c}} &\le \norm{\hat{g}}_{\mathbf{c}} + \sum_{Q \in \mathcal{Q}}\left( c_t\binom{|Q|}{t} \hat{g}(Q) - c_{|Q|}\hat{g}(Q) \right) \le \norm{\hat{g}}_{\mathbf{c}} + c_t \binom{b}{t} \sum_{Q \in \mathcal{Q}} \hat{g}(Q)\\
        &\le \norm{\hat{g}}_{\mathbf{c}} + c_t\binom{b}{t}|E(G) \setminus E(G')|n^{t-2} \\
        &\le \norm{g}_{\mathbf{c}} +  2c_t\frac{n^t}{(b-t)^{t-1}} + c_t\binom{b}{t}|E(G) \setminus E(G')|n^{t-2}. \qedhere
    \end{align*}
\end{proof}

\begin{proof}[Proof of Theorem~\ref{thm:asymptotic_bound_decomposition}]
We may suppose $\rho < 1$. Let $b,b' > t$ be integers such that
\[  \frac{2 c_t}{(b'-t)^{t-1}} \le \frac{\rho}{100} \text{ and } \binom{b'}{t} \frac{100 c_t}{\rho} \le b.\]
Choose $d\in(0,1)$ and a positive integer  $m_0$ such that
\[  1/m_0 + d \le \frac{\rho}{100 c_t}\binom{b'}{t}^{-1}.  \]
 Let $\ep_0(b,d,(c_i)_{i = t}^b,\rho/100)$ and $n_0(b,d,(c_i)_{i = t}^b,\rho/100)$ be as in Lemma~\ref{lem:frac_to_int}. Let $G$ be an $n$-vertex graph such that $n \ge \max\{4b, n_0(1-\ep)^{-1}\}$.  Let  
 \[\ep < \min\left\{\ep_0,\frac{\rho}{100 c_t}\binom{b'}{t}^{-1} \right\}.\]
 By Lemma \ref{lem:no_big_cliques}, there exists $G' \subseteq G$ such that $|V(G')| = n$, $\omega(G') \le b$, and $|E(G) \setminus E(G')| \le n^2/b$. By Lemma~\ref{lem:remove_edges_decomposition} and our choice of $b$, 
 \[  \decompositionopt{t}{G} \le \decompositionopt{t}{G'} + c_t|E(G) \setminus E(G')|n^{t-2} \le \decompositionopt{t}{G'} +  \rho n^t/100.\]
 By Lemma~\ref{lem:cleaning}, $G'$ has a subgraph $\tilde{G}$ such that $\tilde{G}$ is an $(\tilde{n},\ell,\ep,d)$-graph, $\tilde{n} = |V(\tilde{G})| \ge (1 - \ep)n$, and $|E(G')\setminus E(\tilde{G})| \le (2\ep + 1/m_0 + d)n^2$  where $m_0 \le \ell \le M(\ep)$. By Lemma~\ref{lem:remove_edges_decomposition} and our choice of $\varepsilon$, $m_0$, and $d$, 
\[ \decompositionopt{t}{G'} \le \decompositionopt{t}{\tilde{G}} + c_t|V(G') \setminus V(\tilde{G})|n^{t-1} + c_t|E(G') \setminus E(\tilde{G})|n^{t-2} \le  \decompositionopt{t}{\tilde{G}} + 4\rho n^t/100. \]

    As $\ep < \ep_0$ and $\tilde{n} = (1-\ep)n \ge n_0$, by Lemma~\ref{lem:frac_to_int},
    \[ \decompositionopt{t}{\tilde{G}} \le \decompositionoptfrac{t}{\tilde{G}} + \rho \tilde{n}^t/100 \le \decompositionoptfrac{t}{\tilde{G}} + \rho n^t/100.\]
    By Lemma~\ref{lem:remove_edges_decomposition} (with $b'$ serving the role of $b$),
    \begin{align*}
    \decompositionoptfrac{t}{\tilde{G}} &\le \decompositionoptfrac{t}{G'} + \frac{4c_tn^t}{(b'-t)^{t-1}} + c_t\binom{b'}{t}|E(G') \setminus E(\tilde{G})|n^{t-2} + c_t\binom{b'}{t}|V(G') \setminus V(\tilde{G})|n^{t-1}\\
    &\le \decompositionoptfrac{t}{G'} + 6\rho n^t/100 \\
    &\le \decompositionoptfrac{t}{G} + 6\rho n^t/100 + \frac{2c_tn^t}{(b'-t)^{t-1}} + c_t\binom{b'}{t}|E(G) \setminus E(G')|n^{t-2}\\
    &\le \decompositionoptfrac{t}{G} + 8\rho n^t/100.
    \end{align*}
    Combining this sequence of inequalities, we conclude that
    \[ \decompositionopt{t}{G} \le \decompositionoptfrac{t}{G} + \rho n^t.\qedhere\]    
\end{proof}

We conjecture that the additive gap for the $t$-clique cover number is small.
\begin{conjecture}\label{conj:gap}
For every $\ep>0$ and positive integer $t$, there exists $n_0$ such that for every graph $G$ on $n \geq n_0$ vertices
\[ \tcoveropt{G} \leq \tcoveroptfrac{G} + \ep n^t .\]
\end{conjecture}
Certainly Conjecture~\ref{conj:gap} could be made for non-uniform costs, although it is not clear what (mild) conditions on the cost vector are necessary. It is not clear if the cover analogue of Lemma~\ref{lem:remove_edges_decomposition} is true. We note that removing a single edge could increase the cover number by a significant amount: $\coveroptfrac{2}{K_{1,1,n}} = n$, while $\coveroptfrac{2}{K_{2,n}} = 2n$, despite $K_{1,1,n}$ and $K_{2,n}$ differing in only one edge.

\subsection{Lower bounds}

We  show that the additive gap could be as big as $n^{2-o(1)}$ for some graphs. To do this, we first construct a graph with a large multiplicative gap using a construction of Scheinerman and Trenk~\cite{ST}.
\begin{proposition}\label{prop:gap}
For every $t \geq 2$ and even $n \geq 2t$, there exists a graph $G_n$ on $n$ vertices such that $\coveropt{t}{G_n}/\coveroptfrac{t}{G_n}  = \Omega_t(\log(n))$. 
\end{proposition}

\begin{proof}
Let $t \geq 2$ and $n$ be given such that $n$ is even and $n \geq 2t$. Let $G$ be a complete multipartite graph with $n/2$ classes of size $2$ each. By the discussion on~\eqref{LP:cover} and~\eqref{LP:coverdual}, when $c_i = 1$ for all $i$, the optimal value of~\eqref{LP:cover} is $\prod_{i=1}^t x_i$ (see the penultimate paragraph of Section~\ref{sec:LP}). Thus we have $\coveroptfrac{t}{G} = 2^t$. We claim that $\coveropt{t}{G} > \floor{\log_2(n/t)}$.

Let $\{v_i,u_i\}$ for $i \in [n/2]$ be the parts of $G$, and let $f$ be an optimal $t$-clique cover of $G$, i.e., $\norm{f} = \coveropt{t}{G}$, and assume the support of $f$ consists only of maximal cliques of $G$. For each maximal clique $K \in \maximalcliques{G}$ we associate a binary vector $x$ of length $n/2$ such that
\begin{align*}
x_i = \begin{cases}
0 \quad\quad\quad \text{ if } u_i \in K,\\
1 \quad\quad\quad  \text{ if } v_i \in K.
\end{cases}
\end{align*}
By induction, for every $k \ge 1$, for every set $B$ of $k$ binary vectors of length $n/2$, there exists a set of coordinates $L_k$  of size at least $\floor{n/2^{k}}$  such that for every   $x \in B$, every coordinate of  $x$ is the same on $L_k$ (but for $x\ne y \in B$ they could be different).
Let $k = \floor{\log_2(n/t)}$ so that $n/2^{k} \ge t$.

Let $B$ be the set of binary vectors associated to the support of $f$. If $\norm{f} = |B| \leq k$, then there exists $t$ coordinates such that for every  $x \in B$, on those $t$ coordinates $x$ has the same value. In other words,  there is a set of $t$ classes of $G$ on which the support of $f$ covers only two cliques. Since $t \geq 2$, this implies that there is a $t$-clique not covered by $f$, and hence $\coveropt{t}{G} = \norm{f} > k = \floor{\log_2(n/t)}$.
\end{proof}

The gap exhibited in Proposition~\ref{prop:gap} is best possible up to the constant factor. To see this, we can interpret the $t$-clique cover problem as a set cover problem in the following way. We let the ground set be $\tcliques{t}{G}$ and let our sets consist of $\{ T \in \tcliques{t}{G} : T \subseteq K \}$ for each $K \in \cliques{G}$. The classical greedy algorithm of Chv\'{a}tal \cite{C} for the set cover problem shows $\coveropt{t}{G}/\coveroptfrac{t}{G} = O(\log(|\tcliques{t}{G}|))$. As $|\tcliques{t}{G}| \le \binom{n}{t}$ for all $n$-vertex graphs $G$, we conclude $\coveropt{t}{G}/\coveroptfrac{t}{G} = O_t(\log(n))$.

To show that the worst additive gap is at least $n^{2-o(1)}$, we make use of Proposition~\ref{prop:gap} and a strengthening of a classical result of Erd\H{o}s, Frankl, and R\"{o}dl \cite{EFR} due to Mubayi and Verstra\"ete \cite{MV}.

\begin{theorem}[Mubayi and Verstra\"ete  \cite{MV}]\label{thm:EFR}
    For every $R,N \ge 3$ with $N \ge R \ge \log N$, there exists an $N$-vertex $R$-uniform hypergraph $\cH$ with the following properties:
    \begin{enumerate}[(i)]
        \item $|E(\cH)| \ge N^2/R^{8\sqrt{\log_R N}}$.
        \item $\cH$ is linear, i.e., for every pair of distinct edges $e,f \in \cH$ we have $|e \cap f| \le 1$.
        \item $\cH$ is triangle-free, that is, for every three distinct edges, $e,f,g \in \cH$, if $|e \cap f| = |f \cap g| = |g \cap e| = 1$, then $|e \cap f \cap g| = 1$.
    \end{enumerate}
\end{theorem}

Now we are ready to present our graph with large additive gap.

\begin{theorem}\label{thm:additive_lower_bound}
  For every $t\ge 2$ there is  some constant $c > 0$ such that  for every $n \ge e^4+t$ there exists a graph $G_n$ such that
    \[ \coveropt{t}{G_n} - \coveroptfrac{t}{G_n} > \frac{c \log\log n}{\left(\log n\right)^{8 \sqrt{\frac{\log n}{\log \log n}}}} n^2=n^{2-o(1)}.\]
    
\end{theorem}

\begin{proof} First, observe that it is sufficient to prove the statement for $t=2$: given such an example on $n-t+2$ vertices, we can add a clique on $t-2$ vertices, joined with a complete bipartite graph to the rest of the  $n-t+2$ vertices. Hence, we assume now that $t=2$.
    As $\log n \ge 4$, by Proposition~\ref{prop:gap}, there exists a graph $H_n$ on $2\ceil{\log(n)/2}$ vertices  and a constant $c' > 0$, such that $\coveropt{2}{H_n} \ge c' \log\log n \cdot  \coveroptfrac{2}{H_n}$. Let $\cH$ be a hypergraph as in Theorem~\ref{thm:EFR}, with $N = n$ and $R = \ceil{\log n}$. Define the graph
    $G_n = \bigcup_{S \in E(\cH)} H_S$, where for each $S \in E(\cH)$, $H_S$ is an isomorphic copy of $H_n$ such that $V(H_S) = S$.

    The linearity of  $\cH$ implies that  for every $e \in E(G_n)$ we have  $e \in E(H_S)$ for exactly one $S \in E(\cH)$. As $\cH$ is triangle-free, for all cliques $Q \subseteq G_n$, $Q \subseteq H_S$ for exactly one $S \in E(\cH)$. It follows that $\coveropt{2}{G_n} = |E(\cH)| \coveropt{2}{H_n}$ and $\coveroptfrac{2}{G_n} = |E(\cH)| \coveroptfrac{2}{H_n}$. As $|E(\cH)| \ge n^2/R^{8\sqrt{\log_R n}}$, the claim follows.
\end{proof}

\subsection{Packings}

In general, decomposition problems can be relaxed in two directions: covering, as we have addressed up to now, and packing. For a graph $G$, and a (possibly infinite) family of graphs $\mathcal{F}$, let $\mathcal{F}(G)$ be the collection of copies of elements of $\mathcal{F}$ in $G$. The $\mathcal{F}$-packing number, denoted $\nu_{\mathcal{F}}(G)$ is the maximum number of pairwise edge-disjoint elements of $\mathcal{F}(G)$. The fractional packing number $\nu^*_{\mathcal{F}}(G)$ is the maximum $\norm{w}$ where $w : \mathcal{F}(G) \to [0,1]$ such that
\[ \sum_{\substack{e\in H \in \mathcal{F}(G) }} w(H) \le 1\]
for every $e \in G$. Yuster \cite{Y} strengthened an earlier result of Haxell and R\"{o}dl \cite{HR} giving a bound on the additive error of $\nu^*_{\mathcal{F}}(G) - \nu_{\mathcal{F}}(G)$. 
\begin{theorem}\label{thm:Yuster}
    If $\mathcal{F}$ is a fixed family of graphs and $G$ is a graph with $n$ vertices, then $\nu^*_{\mathcal{F}}(G) - \nu_{\mathcal{F}}(G) = o(n^2)$. 
\end{theorem}

Yuster also  gave a lower bound for the additive error by a probabilistic construction in a similar strategy to our proof of Theorem~\ref{thm:additive_lower_bound}.

\begin{proposition}[Yuster \cite{Y}]\label{prop:packing_additive_error_lower_bound}
    For every $\ep > 0$, there exists $t = t(\ep)$ and $N = N(\ep)$ such that for all $n > N$, there exists a graph $G$ with $n$ vertices for which $\nu_{K_t}^*(G) - \nu_{K_t}(G) > n^{2 - \ep}$. 
\end{proposition}

Yuster remarked that there is  an infinite family of graphs such that $\nu_{K_3}^*(G) - \nu_{K_3}(G) = \Theta(n^{1.5})$ and asked if $\nu_{K_3}^*(G) - \nu_{K_3}(G) = O(n^{1.5})$ for all graphs $G$.  We  strengthen Proposition~\ref{prop:packing_additive_error_lower_bound} utilizing Theorem~\ref{thm:EFR}, and show that the integrality gap for packing triangles is even larger.

\begin{theorem}
    For all integers $t \ge 3$ and  $n \ge e^{t+1}$, there exists a graph $G_n$ such that
    \[ \nu^*_{K_t}(G_n) - \nu_{K_t}(G_n) > \frac{c n^2}{\left(\log n\right)^{8 \sqrt{\frac{\log n}{\log \log n}}}}\]
    for some constant $c = c(t) > 0$.
\end{theorem}

\begin{proof}
     Let $\cH$ be a hypergraph as in Theorem~\ref{thm:EFR}, with $N = n$ and $R$ be the smallest integer greater than $\log n$ such that $t+1$ divides $R$. Let $H$ be the disjoint union of $R/(t+1)$ copies of $K_{t+1}$.
     Define the graph $G_n = \bigcup_{S \in E(\cH)} H_S$ where for all $S \in E(\cH)$, $H_S$ is an isomorphic copy of $H$ and $V(H_S) = S$. Note that $$\nu_{K_t}^*(H_S) = \frac{t+1}{t-1}\cdot \frac{R}{t+1}  = \frac{t+1}{t-1}\nu_{K_t}(H_S) .$$
     As $\cH$ is linear, for all $e \in E(G_n)$, we have that $e \in E(H_S)$ for exactly one $S \in E(\cH)$. Furthermore, as $\cH$ is triangle-free, all cliques of size $t+1$ of $G_n$ are contained in exactly one $S \in E(\cH)$. As  $\nu_{K_t}(G_n) = |E(\cH)| \nu_{K_t}(H)$, $\nu^*_{K_t}(G_n) = |E(\cH)| \nu^*_{K_t}(H)$, and $|E(\cH)| \ge n^2/R^{8\sqrt{\log_R n}}$, the claim follows.
\end{proof}


\section*{Acknowledgments} 
This work was partially done when three of the authors (JB, JH, MCW) visited the ECOPRO group at the Institute for Basic Science (IBS), where the project was initiated. We are very grateful for the kind hospitality of IBS and other visitors who participated in fruitful discussions at the beginning of this project. We thank the referees for their careful reading and useful comments that improved the exposition of the paper.

\paragraph{AI Declaration.} AI was only used in the final copyediting, not any math or the bulk of the writing.

\bibliographystyle{amsplain}


\begin{aicauthors}
\begin{authorinfo}[jozsi]
  J\'ozsef Balogh\\
  University of Illinois Urbana--Champaign\\
  Urbana, IL, USA\\
  jobal\imageat{}illinois\imagedot{}edu \\
  \url{https://sites.google.com/view/jozsefbaloghmath}
\end{authorinfo}
\begin{authorinfo}[jialin]
  Jialin He\\
  School of Mathematical Sciences, Key Laboratory of MEA (Ministry of Education), Shanghai Key Laboratory of PMMP, and Nantong Institute for Applied Mathematics and Artificial Intelligence \\
  East China Normal University \\
  Shanghai 200241, China \\
  jlhe\imageat{}math\imagedot{}ecnu\imagedot{}edu\imagedot{}cn
\end{authorinfo}
\begin{authorinfo}[bob]
  Robert A. Krueger\\
  Carnegie Mellon University\\
  Pittsburgh, PA, USA\\
  rkrueger\imageat{}andrew\imagedot{}cmu\imagedot{}edu\\
  \url{https://bob1123.github.io/}
\end{authorinfo}
\begin{authorinfo}[the]
  The Nguyen\\
  University of Illinois Urbana--Champaign\\
  Urbana, IL, USA \\
  thenguyen\imagedot{}bitalgo\imageat{}gmail\imagedot{}com
\end{authorinfo}
\begin{authorinfo}[mike]
  Michael C. Wigal\\
  University of Illinois Urbana--Champaign\\
  Urbana, IL, USA\\
  michael\imagedot{}wigal\imagedot{}math\imageat{}gmail\imagedot{}com
\end{authorinfo}
\end{aicauthors}

\end{document}